\documentclass[11pt]{amsart}
\usepackage[utf8]{inputenc}
\usepackage{amssymb,xcolor,comment}
\usepackage[shortlabels]{enumitem}
\usepackage{mathdots}
\usepackage[a4paper, footskip=1cm, headheight = 16pt, top=3.5cm, bottom=2.5cm,  right=2.5cm,  left=2.5cm ]{geometry}
\usepackage[colorlinks,linkcolor=blue,citecolor=red]{hyperref}
\usepackage[center]{caption}
\usepackage{eufrak}
\usepackage{dsfont}

\newcommand{\previousauthors}{Abrishami, Alecu, Chudnovsky, Hajebi, Spirkl, Vu\v{s}kovi\'c \cite{abrishami2022induced}}

\newcommand{\pictured}{Dashed lines represent paths of arbitrary length.}

\newcommand{\whalf}{$\left(w, \frac{1}{2}\right)$}

\usepackage{listings}
\newtheorem{theorem}{Theorem}[section]
\newtheorem{lemma}[theorem]{Lemma}
\newtheorem{conjecture}[theorem]{Conjecture}

\newtheorem{assumptions}[theorem]{Assumption}

\usepackage[foot]{amsaddr}

\usepackage[T1]{fontenc}

\usepackage[leqno]{amsmath}

\makeatletter
\newcommand{\leqnomode}{\tagsleft@true}
\newcommand{\reqnomode}{\tagsleft@false}
\makeatother

\usepackage{latexsym}
\usepackage{amsfonts}

\usepackage{tikz}
\usepackage{float}
\usepackage{lmodern}

\usepackage{marvosym}               

\DeclareMathOperator{\tw}{tw}

\DeclareMathOperator{\core}{core}
\DeclareMathOperator{\closure}{closure}

\newcounter{tbox}

\newcommand{\sta}[1]{\medskip\medskip\refstepcounter{tbox}\noindent{\parbox{\textwidth}{(\thetbox) \emph{#1}}}\vspace*{0.3cm}}

\newcommand{\mylongtitle}[1]{%
  \ifodd\value{page}%
    \protect\parbox{0.97\linewidth}{#1}\hfill%
  \else%
    \hfill\protect\parbox{0.97\linewidth}{#1}%
  \fi%
}

\usepackage{latexsym}
\usepackage{amsfonts}
\usepackage[leqno]{amsmath}
\usepackage{tikz}
\usetikzlibrary{decorations.pathmorphing}
\usetikzlibrary{arrows,positioning}

\tikzset{snake it/.style={decorate, decoration=snake}}
\usepackage{float}
\usepackage{lmodern}
\usepackage[T1]{fontenc}

\usepackage{marvosym}

\newcommand{\otherlabel}[2]{\protected@edef\@currentlabel{#2}\label{#1}}
\makeatother

\mathchardef\mh="2D

\title[Induced subgraphs and tree decompositions XIV.]{Induced subgraphs and tree decompositions\\
XIV. Non-adjacent neighbours in a hole}

\author{Maria Chudnovsky$^{\ast \amalg}$}
\author{Sepehr Hajebi$^{\mathsection}$}
\author{Sophie Spirkl$^{\mathsection \parallel}$}

\address{$^{\ast}$Princeton University, Princeton, NJ, USA}
\address{$^{\mathsection}$Department of Combinatorics and Optimization, University of Waterloo, Waterloo, Ontario, Canada}
\address{$^{\amalg}$ Supported by NSF-EPSRC Grant DMS-2120644 and by AFOSR grant FA9550-22-1-0083.} 
\address{$^{\parallel}$ We acknowledge the support of the Natural Sciences and Engineering Research Council of Canada (NSERC), [funding reference number RGPIN-2020-03912].
Cette recherche a \'et\'e financ\'ee par le Conseil de recherches en sciences naturelles et en g\'enie du Canada (CRSNG), [num\'ero de r\'ef\'erence RGPIN-2020-03912]. This project was funded in part by the Government of Ontario. This research was conducted while Spirkl was an Alfred P. Sloan Fellow.}
\date {\today}
\begin{document}
\maketitle

\begin{abstract}
    A \emph{clock} is a graph consisting of an induced cycle $C$ and a vertex not in $C$ with at least two non-adjacent neighbours in $C$. We show that every clock-free graph of large treewidth contains a ``basic obstruction'' of large treewidth as an induced subgraph: a complete graph, a subdivision of a wall, or the line graph of a subdivision of a wall. 
\end{abstract}

\section{Introduction}
All graphs in this paper are finite, undirected and simple. Given a graph $G$ and a set $X \subseteq V(G)$, we write $G \setminus X$ for the graph arising from $G$ by deleting all vertices in $X$, and we write $G[X]$ for the subgraph of $G$ induced by $X$, that is, the graph $G \setminus (V(G) \setminus X))$. If $H$ is isomorphic to $G[X]$ for some $X$, we say that $G$ \emph{contains} $H$; otherwise, we say $G$ is \emph{$H$-free}. For a family $\mathcal{H}$ of graphs, we say that $G$ is \textit{$\mathcal{H}$-free} if $G$ is $H$-free for all $H \in \mathcal{H}$. 

A \emph{tree decomposition} of a graph $G$ is a pair $(T, \tau)$ where $T$ is a tree and $\tau : V(T) \rightarrow 2^{V(G)}$ assigns to each vertex of $T$ a subset of $V(G)$, such that the following hold: 
\begin{itemize}
    \item $\bigcup_{t \in V(T)} \tau(t) = V(G)$; 
    \item for every edge $xy \in E(G)$, there is a vertex $t \in V(T)$ such that $x, y \in \tau(t)$; and
    \item for every $v \in V(G)$, the induced subgraph $T[\{t \in V(T): v \in \tau(t)\}]$ is connected. 
\end{itemize}
The \emph{width} of $(T, \tau)$ is $\max_{t \in V(T)} |\tau(t)|-1$. The \emph{treewidth of $G$}, denoted $\tw(G)$, is the minimum width of a tree decomposition of $G$. 

Treewidth was defined and used by Robertson and Seymour \cite{robertson1986graph} as part of the Graph Minors series. In particular, from the point of view of graph minors, as well as subgraphs, it is well-known that (subdivided) walls ``cause'' large treewidth \cite{robertson1986graph5}. 

In the realm of induced subgraphs, this causal role is partly played by the four natural families of graphs: 
\begin{itemize}
    \item the complete graph $K_{t+1}$; 
    \item the complete bipartite graph $K_{t,t}$; 
    \item subdivisions of the $(t\times t)$-wall; and
    \item line graphs of subdivisions of the $(t\times t)$-wall.
\end{itemize}
These graphs, shown in Figure~\ref{fig:basic} and defined in \cite{abrishami2021induced}, are called \emph{$t$-basic obstructions}. For every $t\geq 1$, all $t$-basic obstructions have treewidth $t$, and so if a graph $G$ contains a $t$-basic obstruction, then $\tw(G) \geq t$. The converse, on the other hand, is not true: Let us call a graph \emph{$t$-clean} if it does not contain a $t$-basic obstruction. Each of the following constructions are examples of $3$-clean graphs of arbitrarily large treewidth: 
\begin{itemize}
    \item Pohoata-Davies graphs \cite{davies, pohoata2014unavoidable} (see Figure~\ref{fig:davies}); 
    \item ``Layered wheels'' \cite{sintiari2021theta};
    \item ``Occultations'' \cite{twix,bonamy2023sparse}. 
\end{itemize}

  \begin{figure}[t!]
        \centering
        \includegraphics[width=0.8\textwidth]{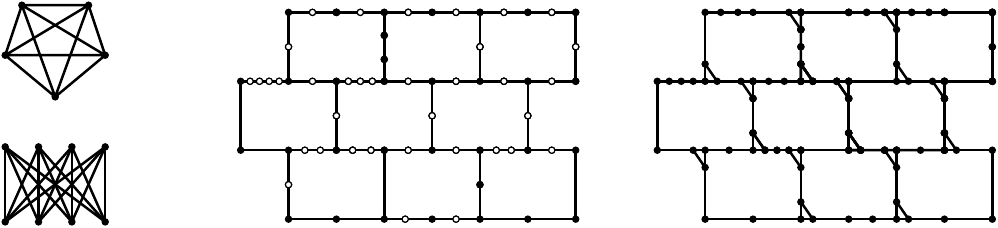}
        \caption{The $4$-basic obstructions}
        \label{fig:basic}
    \end{figure}   

Let us say that a class $\mathcal{C}$ of graphs is \emph{clean} if there is a function $f : \mathbb{N} \rightarrow \mathbb{N}$ such that for all $t$, every $t$-clean graph $G$ in $\mathcal{C}$ satisfies $\tw(G) \leq f(t)$. Then, as we saw above, the class of all graphs is not clean. Indeed, several proper hereditary classes are known not to be clean, among which the class of wheel-free graphs is rather well-studied.

A \textit{hole} is an induced cycle on four or more vertices. A \textit{wheel} is a graph consisting of a hole $C$ and a vertex $v$ with at least three neighbours in $C$. Wheels appear to be essential to the construction of two of the above three ``non-basic'' obstructions: the layered wheels (hence the name) and the occultations. In fact, these two constructions contain a hole with a vertex that has arbitrarily many neighbours in it. The Pohoata-Davies graphs, however, are wheel-free, and so the class of wheel-free graphs is not clean. But the Pohoata-Davies graphs conatin a relaxed version of wheels: A \emph{clock} is a graph consisting of a hole $C$ and a vertex $v$, called the \emph{center} of the clock, such that the neighbours of $v$ in $C$ contain two non-adjacent vertices. Observe that every wheel is a clock, and that clocks are found in abundance in the Pohoata-Davies graphs. 

In the present paper, our main result is the following (conjectured in \cite{abrishami2022induced}): 
\begin{theorem}\label{thm:main}
The class of clock-free graphs is clean. 
\end{theorem}

In \cite{abrishami2022induced}, with Abrishami, Alecu, and Vu\v{s}kovi\'c, we proved  a weakening of Theorem~\ref{thm:main}, that graphs in which every vertex $v$ has at most one neighbour in every hole not containing $v$ form a clean class. Explicitly, we proved the following (prisms and pyramids are defined in Section \ref{sec:defs}): 
\begin{theorem}[\previousauthors]
    The class of (clock, prism, pyramid)-free graphs is clean. 
\end{theorem}

On the other hand, the following strengthening of Theorem~\ref{thm:main} might be true (where a \textit{$t$-clock} is a clock consisting of a hole $C$ and a vertex $v$ such that $v$ there are two neighbours $x,y\in V(C)$ of $v$ where the distance between $x$ and $y$ along $C$ is at least $t$): 
\begin{conjecture}
    For every fixed $t\geq 1$, the family of $t$-clock-free graphs is clean. 
\end{conjecture}

\section{Definitions} \label{sec:defs}

We begin with some definitions that will be used throughout the paper. For ease of notation, we use graphs and their vertex sets interchangeably. Given a path $P$, we refer to its vertices of degree at most one as the \emph{ends} of $P$, and we denote by $P^*$ the \emph{interior} of $P$, that is, the set obtained from $P$ by deleting the ends of $P$. If $x, y$ are the ends of $P$, we also say that $P$ is a \emph{path from $x$ to $y$}. The \emph{length} of a path is its number of edges. A \emph{prism} is a graph consisting of two triangles with disjoint vertex sets $\{a_1, a_2, a_3\}$ and $\{b_1, b_2, b_3\}$, as well as three paths $P_1, P_2, P_3$ such that: 
\begin{itemize}
    \item $P_i$ has ends $a_i$ and $b_i$ for all $i \in \{1, 2, 3\}$; and
    \item for distinct $i, j \in \{1, 2, 3\}$, the only edges between $P_i$ and $P_j$ are the edges $a_ia_j$ and $b_ib_j$. 
\end{itemize}
A \emph{pyramid} is a graph consisting of a vertex $a$ (called the \emph{apex}) and a triangle with vertex set $\{b_1, b_2, b_3\}$ (called the \emph{base}), as well as three paths $P_1, P_2, P_3$ such that: 
\begin{itemize}
    \item $P_i$ has ends $a$ and $b_i$ for all $i \in \{1, 2, 3\}$; 
    \item for $i \in \{1, 2, 3\}$, each path $P_i$ has length at least one, and there is at most one $i \in \{1, 2, 3\}$ such that the path $P_i$ has length exactly one; and
    \item for distinct $i, j \in \{1, 2, 3\}$, the only edge between $P_i \setminus \{a\}$ and $P_j \setminus \{a\}$ is the edge $a_ia_j$. 
\end{itemize}
A \emph{short pyramid} is a pyramid in which one of $P_1, P_2, P_3$ has length exactly one. We observe: 
\begin{lemma}\label{lem:short}
    If $G$ is clock-free, then $G$ does not contain a short pyramid. 
\end{lemma}
\begin{proof}
    Notice that a short pyramid $Q$ is a clock: with $P_1, P_2, P_3$ and $a, b_1, b_2, b_3$ as in the definition of a pyramid, let us assume that $P_1$ has length exactly one. Then $Q \setminus \{b_1\}$ is a hole, and $b_1$ has three neighbours in it, including $a$ and $b_2$, which are non-adjacent (since $P_2$ has length more than one). 
\end{proof}
A \emph{theta} is a graph consisting of two non-adjacent vertices $a$ and $b$ (called its \emph{ends}), as well as three paths $P_1, P_2, P_3$ such that: 
\begin{itemize}
    \item $P_i$ has ends $a$ and $b$ for all $i \in \{1, 2, 3\}$; and
    \item for distinct $i, j \in \{1, 2, 3\}$, there are no edges between $P^*_i$ and $P^*_j$. 
\end{itemize}
A graph is a \emph{three-path configuration} if it is a prism, a pyramid, or a theta. In each case, we refer to $P_1, P_2, P_3$ as the \emph{paths} of the three-path-configuration. 

For a graph $G$ and a vertex $v \in V(G)$, we write $N_G(v)$ for the set of neighbours of $v$ in $G$, omitting the subscript when there is no danger of confusion. We write $N_G[v]$ for the set $N_G(v) \cup \{v\}$. For $X \subseteq V(G)$, we write $N_G(X) = \bigcup_{x \in X} N_G(x) \setminus X$ and $N_G[X] = X \cup N_G(X)$. 

Given a graph $G$, a set $X \subseteq V(G)$ is a \emph{cutset} of $G$ if $G \setminus X$ is non-empty and not connected. A \emph{clique cutset} is a cutset which is a clique. A \emph{star cutset} is a cutset $X$ such that there is a vertex $x \in X$ with $X \subseteq N[x]$. 

Given $X, Y \subseteq V(G)$, we say that $X$ \emph{is anticomplete to} $Y$ if there is no edge in $G$ with one end in $X$ and the other in $Y$, and we say $x\in V(G)$ \textit{is anticomplete to} $Y$ if $\{x\}$ is anticomplete to $Y$. A \emph{separation} of $G$ is a triple $(A, C, B)$ of pairwise disjoint subsets of $V(G)$ with union $V(G)$ such that $A$ is anticomplete to $B$.

      \begin{figure}[t!]
        \centering
        \includegraphics[width=0.3\textwidth]{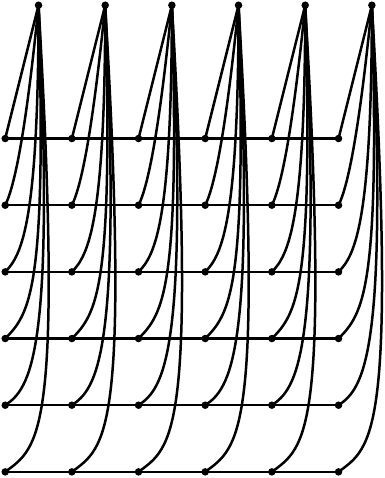}
        \caption{A graph from the Pohoata-Davies construction}
        \label{fig:davies}
    \end{figure} 

\section{Diamonds}\label{sec:diamond}

A \emph{diamond} is a four-vertex graph with exactly five edges. In this section, we use the following results from \cite{abrishami2022induced} to reduce Theorem~\ref{thm:main} to the diamond-free case. 

\begin{lemma}[\previousauthors]
  \label{startoclique}
  Let $G$ be a clock-free graph and let $(A,C,B)$ be a separation of $G$
  with $A \neq \emptyset$ and $B \neq \emptyset$. Suppose that there
  exist $v_1, \dots, v_k \in C$ such that $C \subseteq \bigcup_{i=1}^k N[v_i]$.
  Let $D_1$ be a component of $A$ and let $D_2$ be a component of $B$.
  Then there exist cliques $X_1, \dots, X_k \subseteq C$ of $G$ such that every path
  from a
  vertex in $D_1$ to a vertex in $D_2$ has a vertex in $\bigcup_{i=1}^k X_i$. In particular,
  if $G$ admits a star cutset, then $G$ admits a clique cutset.
\end{lemma}

\begin{lemma}[\previousauthors]
  \label{diamondlemma}
  Let $G$ be a clock-free graph  and assume that $G$ does not admit a star cutset.  Then $G$ is diamond-free. 
\end{lemma}

Combining Lemmas \ref{startoclique} and \ref{diamondlemma}, we conclude that every clock-free graph which contains a diamond has a clique cutset; and moreover, every clock-free graph which contains a star cutset has a clique cutset. Since clique cutsets do not affect treewidth (see Lemma 7 in \cite{bodlaender2006safe}), to prove Theorem \ref{thm:main}, it now suffices to prove the following: 

\begin{theorem}\label{thm:mainnew}
The family of (clock, diamond)-free graphs with no star cutset is clean. 
\end{theorem}

We conclude with a simple observation about diamond-free graphs: 
\begin{lemma}\label{lem:diamond}
    Let $G$ be diamond-free. Then the following hold: 
    \begin{itemize}
        \item For every $v \in V(G)$, the graph $G[N(v)]$ is a disjoint union of cliques pairwise anticomplete to each other. 
        \item For every edge $xy \in E(G)$, there is exactly one maximal clique of $G$ containing $\{x, y\}$.
    \end{itemize}
\end{lemma}
\begin{proof}
    The first bullet point follows from observing that $G[N(v)]$ does not contain an induced two-edge path; the second follows by observing that $N(x) \cap N(y)$ is a clique. 
\end{proof}

\section{Paws and seagulls}

A \emph{paw} is a graph with vertex set $\{a, a', u, v\}$ and edge set $\{aa', av, a'v, uv\}$. A \emph{seagull} is a graph with vertex set $\{a, u, v\}$ and edge set $\{au, av\}$. 

Our goal in this section is to show that paws and seagulls give rise to particularly nice cutsets in (clock, diamond)-free graphs. To this end, we first show that paws and seagulls are contained in some three-path configuration in a prescribed way; then we show that choosing the right one of these three-path configurations indicates the location of the cutset we are looking for. 

We require the following folklore result that appeared, for example, in \cite{abrishami2021submodular}:

\begin{lemma}\label{minimalconnected}
Let $x_1, x_2, x_3$ be three distinct vertices of a graph $G$. Assume that $H$ is a connected induced subgraph of $G \setminus \{x_1, x_2, x_3\}$ such that $V(H)$ contains at least one neighbour of each of $x_1$, $x_2$, $x_3$, and that $V(H)$ is minimal subject to inclusion. Then, one of the following holds:
\begin{enumerate}[(i)]
\item For some distinct $i,j,k \in  \{1,2,3\}$, there exists $P$ that is either a path from $x_i$ to $x_j$ or a hole containing the edge $x_ix_j$ such that
\begin{itemize}
\item $V(H) = V(P) \setminus \{x_i,x_j\}$, and
\item either $x_k$ has two non-adjacent neighbours in $H$ or $x_k$ has exactly two neighbours in $H$ and its neighbours in $H$ are adjacent.
\end{itemize}

\item There exists a vertex $a \in V(H)$ and three paths $P_1, P_2, P_3$, where $P_i$ is from $a$ to $x_i$, such that 
\begin{itemize}
\item $V(H) = (V(P_1) \cup V(P_2) \cup V(P_3)) \setminus \{x_1, x_2, x_3\}$, and 
\item the sets $V(P_1) \setminus \{a\}$, $V(P_2) \setminus \{a\}$ and $V(P_3) \setminus \{a\}$ are pairwise disjoint, and
\item for distinct $i,j \in \{1,2,3\}$, there are no edges between $V(P_i) \setminus \{a\}$ and $V(P_j) \setminus \{a\}$, except possibly $x_ix_j$.
\end{itemize}

\item There exists a triangle $a_1a_2a_3$ in $H$ and three paths $P_1, P_2, P_3$, where $P_i$ is from $a_i$ to $x_i$, such that
\begin{itemize}
\item $V(H) = (V(P_1) \cup V(P_2) \cup V(P_3)) \setminus \{x_1, x_2, x_3\} $, and 
\item the sets $V(P_1)$, $V(P_2)$ and $V(P_3)$ are pairwise disjoint, and
\item for distinct $i,j \in \{1,2,3\}$, there are no edges between $V(P_i)$ and $V(P_j)$, except $a_ia_j$ and possibly $x_ix_j$.
\end{itemize}
\end{enumerate}
\end{lemma}

We use Lemma \ref{minimalconnected} to find three-path-configurations, as follows: 
\begin{lemma}\label{lem:3pc}
    Let $G$ be a (clock, diamond)-free graph with no star cutset. Let $v \in V(G)$, and let $x_1, x_2, x_3 \in N(v)$. If $\{x_1, x_2, x_3\}$ is not a clique of $G$, then $G$ contains a three-path-configuration $Q$ with $v, x_1, x_2, x_3 \in Q$. 
\end{lemma}
\begin{proof}
    Since $G$ is diamond-free and  $G[\{v, x_1, x_2, x_3\}]$ is not isomorphic to $K_4$, we conclude that $G[\{x_1, x_2, x_3\}]$ contains at most one edge. 
    
    Since $N[v] \setminus \{x_1, x_2, x_3\}$ is not a star cutset in $G$, it follows that $V(G) \setminus N[v] \neq \emptyset$. Since $N[v]$ is not a star cutset in $G$, it follows that $D = G \setminus N[v]$ is connected (and non-empty). Since $\{v\} \cup N(D)$ is not a star cutset in $G$, we conclude that $N(D) = N(v)$ and so $x_1, x_2, x_3 \in N(D)$.

    Let $H$ be a minimal induced subgraph of $D$ such that $\{x_1, x_2, x_3\} \subseteq N(H)$. Then $H$ satisfies one of the outcomes of Lemma \ref{minimalconnected}. Note that outcomes (ii) and (iii) of Lemma \ref{minimalconnected} give us the desired three-path configuration. In the case of outcome (i), with $i, j, k$ and $P$ as in Lemma \ref{minimalconnected}, either $P$ or $P\cup \{v\}$ is a hole in $G$; and since $G$ is clock-free, it follows that $x_k$ does not have two non-adjacent neighbours in $P$. Therefore, $x_k$ has exactly two neighbours in $P$, and they are adjacent. Since $G$ is diamond-free, and $x_k$ has a neighbour in $P$, it follows that $x_k$ is non-adjacent to $x_i, x_j$. Therefore, $P \cup \{v, x_k\}$ is a pyramid (if $x_ix_j \not\in E(G)$) or a prism (if $x_ix_j \in E(G)$) in $G$. This completes the proof.
\end{proof}

For disjoint sets $X, Y, Z \subseteq V(G)$, we say that $X$ \emph{separates} $Y$ from $Z$ if every path $P$ with one end in $Y$ and the other in $Z$ satisfies $P \cap X \neq \emptyset$. 

\begin{theorem}\label{thm:paw}
    Let $G$ be a (clock, diamond)-free graph with no star cutset. Suppose that $G$ contains a paw; and let $a, a', u,v \in V(G)$ be as in the definition of a paw. Then there is a vertex $b$ in $G \setminus N(a)$ and a clique $K \subseteq N[b]$ such that $\{v\} \cup K$ separates $\{u\}$ from $\{a, a'\}$. 
\end{theorem}

\begin{proof}
    Let $A$ be the component of $N(v)$ containing $a$ and $a'$. By Lemma \ref{lem:diamond}, it follows that $A$ is a clique. In particular, $u \not\in A$, and so for all distinct $a^*, a^{**} \in A$, we obtain a paw with vertex set $\{a^*, a^{**}, u, v\}$ such that $u$ has degree one and $v$ has degree three in the paw. Since $G$ is diamond-free, no vertex in $V(G) \setminus N[v]$ has more than one neighbour in $A$. 

    By Lemma \ref{lem:3pc}, for all distinct $a^*, a^{**} \in A$, there is a three-path-configuration in $G$ which contains all of $a^*, a^{**}, u$ and $v$. Let $\mathcal{Q}_{a^*, a^{**}}$ be the set of all such three-path-configurations. Since $\{a^*, a^{**}, v\}$ is a clique, it follows that every $Q \in \mathcal{Q}_{a^*, a^{**}}$ is a prism or a pyramid. Moreover, for $Q \in \mathcal{Q}_{a^*, a^{**}}$, we define $P_1(Q), P_2(Q), P_3(Q)$ to be the paths of $Q$ such that $v \in P_1(Q), a^* \in P_2(Q)$ and $a^{**} \in P_3(Q)$. Since $v$ is in a triangle of $Q$, it follows that $v$ is an end of $P_1(Q)$, and $u \in P_1(Q)\setminus \{v\}$. Let us define $b(Q)$ as the end of $P_3(Q)$ which is not equal to $a^{**}$; that is, $P_3(Q)$ is a path with ends $a^{**}$ and $b(Q)$. Let us define $l(Q) = |P_1(Q) \cup b(Q)|$. See Figure \ref{fig:pawproof}. 

    \begin{figure}[t]
        \centering
        \includegraphics[width=0.8\textwidth]{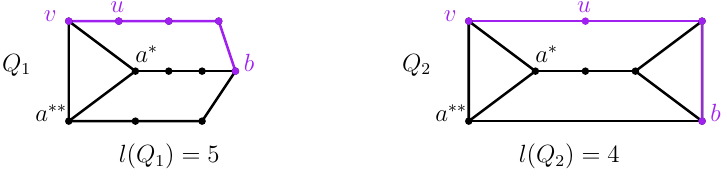}
        \caption{In the proof of Theorem \ref{thm:paw}, the choice of $b = b(Q_i)$ and a path with $l(Q_i)$ vertices of the form $P_1(Q_i) \cup \{b\}$ in the case that $Q_i$ is a pyramid ($i=1$) or $Q_i$ is a prism $(i=2)$.}
        \label{fig:pawproof}
    \end{figure}    

    Let $\mathcal{Q}$ be the union of all sets $\mathcal{Q}_{a^*, a^{**}}$ for distinct $a^*, a^{**} \in A$. Let us pick $Q \in \mathcal{Q}$ with $l(Q)$ minimum; and let $a^*, a^{**} \in A$ such that $Q \in \mathcal{Q}_{a^*, a^{**}}$. By changing the labels of $P_2(Q)$ and $P_3(Q)$ if necessary, we may assume that $a^{**} \neq a$. We claim that $b = b(Q)$ has the desired properties. If $Q$ is a pyramid, then, by Lemma \ref{lem:short}, it follows that $b$ is non-adjacent to both $a^*, a^{**}$. If $Q$ is a prism, then $b$ is non-adjacent to $a^*$ since $b \in P_3(Q) \setminus \{a^{**}\}$. Moreover, $b$ has no neighbour in $A \setminus \{a^*, a^{**}\}$; for suppose otherwise, letting $\hat{a}$ be such a neighbour. Then $\hat{a}$ has two non-adjacent neighbours ($v$ and $b$) in the hole $P_1(Q) \cup P_3(Q)$, a contradiction. Therefore, the only possible neighbour of $b$ in $A$ is $a^{**}$, and in particular, since we relabeled $P_2(Q)$ and $P_3(Q)$ if necessary, it follows that $b$ is non-adjacent to $a$. 
    
    Next, we need to show that $b \neq u$. Suppose that $b=u$. Since $b \in P_3(Q)$ and $u \in P_1(Q)$, it follows that $P_3(Q) \cap P_1(Q) \neq \emptyset$, and hence $Q$ is a pyramid with apex $b=u$. But then, $P_1(Q) = \{u, v\}$, and so $Q$ is a short pyramid, contrary to Lemma \ref{lem:short}. 

    In what follows, we will show: 
    
    \sta{\label{st:pawb}In $G$, the set $X = \{v\} \cup (N[b] \setminus (A \cup \{u\}))$ separates $\{u\}$ from $A$.}
    
    Let us first show that \eqref{st:pawb} implies the statement of the theorem: By Lemma \ref{startoclique} applied to $G \setminus \{v\}$, it follows that $G \setminus \{v\}$ has a clique cutset $K$ contained in $N[b]$; but then $K \cup \{v\}$ is the desired cutset of $G$. 

    It remains to prove \eqref{st:pawb}. Suppose that \eqref{st:pawb} does not hold. Then $G \setminus X$ contains a path from $Y=P_1(Q) \setminus \{v, b\}$ to $Z=A \cup (P_2(Q) \cup P_3(Q)) \setminus N[b])$ with interior disjoint from $X$; let $R$ be a shortest such path. 
    
    \sta{\label{st:rnonempty}The path $R^*$ is non-empty.}
    
    Suppose not; then $R$ consists of an edge $yz$ with $y \in Y$ and $z \in Z$. Since $Y = P_1(Q) \setminus \{b,v\}$ is anticomplete to $Q \setminus P_1(Q)$, it follows that $z \in Z \setminus Q = A \setminus \{a^*, a^{**}\}$. But then $z$ is adjacent to two non-adjacent vertices, namely $a^*$ and $y$, in the hole $P_1(Q) \cup P_2(Q)$, which violates the assumption that $G$ is clock-free. This proves \eqref{st:rnonempty}. 

    \sta{\label{st:rdisjoint}The path $R^*$ is disjoint from $N[b]$.}
    
    Suppose that there is a vertex $r \in R^* \cap N[b]$. Since $r \not\in X$, it follows that $r \in A \cup \{u\}$. This contradicts the fact that $R^*$ is disjoint from $Y \cup Z$, and proves \eqref{st:rdisjoint}. 

    \medskip

    Let $r_1, \dots, r_t$ be the vertices of $R^*$ in order, such that $r_1$ has a neighbour in $Y$ and $r_t$ has a neighbour in $Z$. There are three vertices in $Q$ which may have neighbours in $R^* \setminus \{r_1, r_t\}$: the vertex $v$, the neighbour $b_2$ of $b$ in $P_2(Q)$, and the neighbour $b_3$ of $b$ in $P_3(Q)$. Let us write $b_1$ for the end of $P_1(Q)$ which is not equal to $v$; so $b_1 = b$ if $Q$ is a pyramid, and $b_1 \in N(b)$ if $Q$ is a prism. See Figure \ref{fig:pawproof2}. By considering the holes $P_i(Q) \cup P_j(Q)$ for distinct $i, j \in \{1, 2, 3\}$, we conclude that $N(x) \cap Q$ is a clique for all $x\in V(G) \setminus Q$. 

    \sta{\label{st:cleanpath} There exists $i \in \{2, 3\}$ such that $r_t$ has no neighbour in $P_i(Q)$.}

    If $N(r_t) \cap Q = \{a^*, a^{**}\}$, then $r_t \in N(v)$ as $G$ is diamond-free, and so $r_t \in A$; however, this contradicts that $R^*$ is disjoint from $Z$. Since $N(r_t) \cap Q$ is a clique not equal to $\{a^*, a^{**}\}$, the only possibility is that $r_t$ is adjacent to $b$; but this contradicts \eqref{st:rdisjoint} and thus proves \eqref{st:cleanpath}. 

    \medskip

    Let $i$ be as in \eqref{st:cleanpath}, and let $j \in \{2, 3\} \setminus \{i\}$. We will fix $i$ and $j$ throughout the remainder of the proof. From the choice of $R$, it follows that $r_t$ has a neighbour in $A \cup (P_j(Q) \setminus N[b])$; let $R'$ be a path from $r_t$ to $\hat{a} \in \{a^*, a^{**}\} \cap P_i(Q)$ with interior in $A \cup (P_j(Q) \setminus N[b])$. Note that the neighbour of $\hat{a}$ in $R'$ is a vertex in $A$. 

    \begin{figure}[t]
        \centering
        \includegraphics[width=0.8\textwidth]{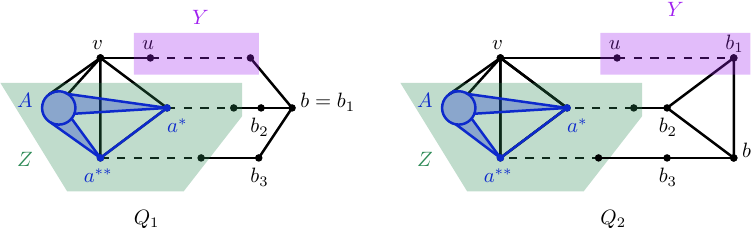}
        \caption{In the proof of Theorem \ref{thm:paw}, the sets $A, Y$ and $Z$ in the case that $Q_i$ is a pyramid ($i=1$) or $Q_i$ is a prism $(i=2)$. \pictured}
        \label{fig:pawproof2}
    \end{figure}  

    \sta{\label{st:bi}The vertex $b_i$ has a neighbour in $R^*$. Moreover, if $Q$ is a prism and $b_3$ has a neighbour in $R^*$, then, traversing $R^*$ from $r_1$ to $r_t$, the first neighbour of $b_3$  appears at the same time or after the first neighbour of $b_2$.}

    Suppose not. Let us define a hole $H$ and paths $T, T'$ as follows:  
    \begin{itemize}
        \item If $b_i$ has no neighbour in $R^*$, then $H = r_t$-$R'$-$\hat{a}$-$P_i(Q)$-$b_1$-$P_1(Q)$-$r_1$-$R^*$-$r_t$ and $T= r_1$-$R^*$-$r_t$-$(R'\setminus \{\hat{a}\})$ and $T' = P_i(Q)$.
        \item Otherwise, $Q$ is a prism and $b_3$ has a neighbour in $R^*$; let $\hat{R}$ be a path from $r_1$ to $b_3$ with interior in $R^*$. From our assumption, it follows that $\hat{R}$ contains no neighbour of $b_2$. We set $H = r_1$-$\hat{R}$-$b_3$-$P_3(Q)$-$a^{**}$-$a^*$-$P_2(Q)$-$b_2$-$b_1$-$P_1(Q)$-$r_1$ and $T = r_1$-$\hat{R}$-$b_3$-$P_3(Q)$-$a^{**}$ and $T' = b_2$-$P_2(Q)$-$a^*$. 
    \end{itemize}
    In either case, the paths $T, T'$ are disjoint, each having an end in $A$. 
    
    Suppose first that $v$ has a neighbour $q$ in $R^* \cap H$. Then, since $G$ does not contain a clock and $v$ is adjacent to $\hat{a} \in H$, it follows that $q\hat{a} \in E(G)$. It follows from the choice of $R$ that $q = r_t$. But $\hat{a} \in P_i(Q)$, contrary to the choice of $i$ and \eqref{st:cleanpath}. Therefore, $v$ has no neighbour in $R^*$. 

    Now we consider $N(r_1) \cap Q$, which is a clique. We would like to show that $N(r_1) \cap Q \subseteq P_1(Q)$. If $r_1$ has a neighbour in $P_1(Q)^*$, then $N(r_1) \cap Q \subseteq P_1(Q)$, as desired. Otherwise, $r_1$ is adjacent to $b_1$ as well as at least one of $b, b_2$, and $Q$ is a pyramid. Since $G$ is diamond-free, it follows that $r_1$ is adjacent to both $b$ and $b_2$. But this contradicts \eqref{st:rdisjoint}. We conclude that $N(r_1) \cap Q \subseteq P_1(Q)$. 
    
    If $N(r_1) \cap Q = \{b'\}$ is a single vertex, then there is a pyramid $Q'$ with paths $P_1(Q') = b'$-$P_1(Q)$-$v$ (which contains $u$), $P_2(Q') = b'$-$P_1(Q)$-$b_1$-$T'$, and $P_3(Q') = b'$-$r_1$-$T$. It follows that $Q' \in \mathcal{Q}$. Since $P_1(Q') \cup \{b(Q')\} = P_1(Q') \subseteq P_1(Q) \cup \{b\}$ and since $b \not\in P_1(Q')$, we conclude that $P_1(Q') \cup \{b(Q')\}\subsetneq P_1(Q) \cup \{b\}$, and so $l(Q') < l(Q)$, contrary to the choice of $Q$. This implies that $N(r_1) \cap Q = \{b', c\}$, where $b'c \in E(Q)$, and we may assume that $P_1(Q)$ traverses $v, c, b'$ in this order. We note that $b' \neq b$ by \eqref{st:rdisjoint}. Now there is a prism $Q'$ in $G$ with paths $P_1(Q') = c$-$P_1(Q)$-$v$ (which contains $u$), $P_2(Q') = b'$-$P_1(Q)$-$T'$, and $P_3(Q') = T$. It follows that $Q' \in \mathcal{Q}$.  Again, we have that $P_1(Q') \cup \{b(Q')\} = P_1(Q') \cup \{b'\} \subseteq (P_1(Q) \setminus \{b\}) \subsetneq P_1(Q) \cup \{b\}$, contradicting the choice of $Q$. This proves \eqref{st:bi}. 

    \medskip 

    Let $R''$ be a shortest path from $r_t$ to $b_j$ with interior in $(A \setminus P_i(Q)) \cup (P_j(Q) \setminus N[b])$. Since we showed that $N(b) \cap A \subseteq \{a^*, a^{**}\}$, it follows that $N(b) \cap R'' \subseteq (N(b) \cap A \cap P_j(Q)) \cup \{b_j\} \subseteq \{b_j\}$. Let $P = r_1$-$R$-$r_t$-$(R''\setminus \{b_j\})$; so in particular, $N(b) \cap P = \emptyset$. Let $P'$ be the shortest subpath of $P$ containing $r_1$ as well as a neighbour of $b_i$ and a neighbour of $b_j$. Since $R''$ contains a neighbour of $b_j$ and $R^*$ contains a neighbour of $b_i$ by \eqref{st:bi}, the path $P'$ is well-defined. Let $p$ be the end of $P'$ not equal to $r_1$; and let $k \in \{2, 3\}$ be maximum such that $b_k$ is adjacent to $p$ and $b_k$ has no neighbour in $P' \setminus \{p\}$ (such $k$ exists by the choice of $P'$, as otherwise $P' \setminus \{p\}$ would be a better choice than $P'$). Let $k' \in \{2, 3\} \setminus \{k\}$. 
    
    Suppose first that either $Q$ is a pyramid, or $Q$ is a prism and $k=3$. Then, there is a hole $H'$ in $G$, defined as $H' = r_1$-$P_1(Q)$-$b$-$b_k$-$p$-$P'$-$r_1$. The vertex $b_{k'}$ has two non-adjacent neighbours in $H'$, namely $b$ and a neighbour in $P'$. Since $G$ is clock-free, this is a contradiction. 

    It follows that $Q$ is a prism and $k = 2$. Since at least one of $b_2, b_3$ has a neighbour in $R^*$ by \eqref{st:bi}, it follows from the choice of $k$ that $b_3$ has a neighbour in $R^*$. But now \eqref{st:bi} implies that the first neighbour of $b_2$ along $R^*$, traversed from $r_1$ to $r_t$, appears at the same time or before the first neighbour of $b_3$, which implies that $k = 3$, a contradiction. This concludes the proof. 
\end{proof}

Next, we show that certain seagulls lead to similar cutsets as in Theorem \ref{thm:paw}. We start with two lemmas. Given a graph $G$, a vertex $v\in V(G)$ is a \emph{claw center} in $G$ if $N_G(v)$ contains three pairwise non-adjacent vertices. 

\begin{lemma} \label{lem:claw}
    Let $G$ be a clock-free graph and let $P$ be a path in $G$. Let $a$ be an end of $P$ and let $y$ be the neighbour of $a$ in $P$. Let $x, v \in N(a)$ such that $\{x, y, v\}$ is a stable set. Then at least one of $x$ and $v$ is anticomplete to $P \setminus \{a\}$. 
\end{lemma}
\begin{proof}
     Suppose not. We may assume that $P$ is chosen minimal such that $a \in P$ and $P \setminus \{a\}$ contains both a neighbour of $x$ and a neighbour of $v$. Then $P \cup \{x, v\}$ is a clock: let $w \in \{x, v\}$ be chosen such that the only neighbour of $w$ in $P \setminus \{a\}$ is the end of $P$ not equal to $a$. Then $w$-$P$-$a$-$w$ is a hole, and the vertex $q \in \{x, v\} \setminus \{w\}$ has at least two non-adjacent neighbours in it, namely $a$ and a vertex in $P \setminus \{a, y\}$. This is a contradiction, proving Lemma \ref{lem:claw}. 
\end{proof}

\begin{lemma}\label{lem:seagull}
    Let $G$ be a (clock, diamond)-free graph with no star cutset. Suppose that $G$ contains a seagull; and let $a, u, v \in V(G)$ be as in the definition of a seagull. Suppose further that $a$ is a claw center in $G$. Then there is a three-path configuration $Q$ in $G$ such that $a, u, v \in Q$ and $a$ is a claw center in $Q$. 
\end{lemma}

\begin{proof}
    Since $N[v] \setminus \{a, u\}$ is not a star cutset in $G$, it follows that there is a path $P$ in $G$ with ends $a$ and $u$ and with interior disjoint from $N[v]$. Let $x$ be the neighbour of $a$ in $P$. Then $x$ is non-adjacent to $v$ from the choice of $P$. 

     From Lemma \ref{lem:diamond}, it follows that $N(a)$ has at least three components. Let us pick $y \in N(a)$ such that $\{x, y, v\}$ is an independent set.     

     Since $N[a] \setminus \{y\}$ is not a star cutset, it follows that there is a path from $y$ to $P \setminus \{x\}$ with interior disjoint from $N[a].$ Let $R$ be a shortest such path. See Figure~\ref{fig:seagulllemma}.

     Let $r_1, \dots, r_t$ denote the vertices of $R$ in order such that $y = r_1$, and $r_t \in P$. Traversing $P$ from $x$ to $u$, let us $w$ and $w'$ denote the first and last neighbour of $r_{t-1}$ in $P$, respectively. Then, $P' = y$-$R$-$w'$-$P$-$u$ is an induced path (due to the choice of $R$ and $w'$). Applying Lemma \ref{lem:claw} to $P'$ and $a, x, y, v$ implies that not both $x$ and $v$ have a neighbour in $R$. But $v$ does have a neighbour in $R$, namely $u$; so $x$ is anticomplete to $R$. 

     \begin{figure}[t]
         \centering
         \includegraphics[width=0.3\textwidth]{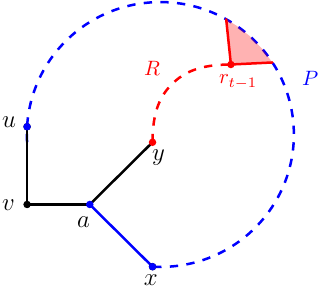}
         \caption{Proof of Lemma \ref{lem:seagull}. \pictured}
         \label{fig:seagulllemma}
     \end{figure}

    Now let us consider the hole $H = a$-$y$-$R$-$w$-$P$-$x$-$a$. Then, since $a \in N(v) \cap H$ and $x, y \not\in N(v)$, it follows that $v$ has no neighbour in $H \setminus \{a\}$ as $G$ is clock-free. 

    Since $a$-$x$-$P$-$u$-$v$-$a$ is a hole, it follows that $r_{t-1}$ either has exactly one neighbour in $P$, or exactly two neighbours in $P$ and they are adjacent. In the former case, we get a theta $Q$ with ends $a$ and $w=w'$ which satisfies \ref{lem:seagull}; in the latter case, we get a pyramid $Q$ with apex $a$ and base $\{w, w', r_{t-1}\}$ which satisfies \ref{lem:seagull}. This concludes the proof. 
\end{proof}

The proof of the next result  follows the same structure as the proof of Theorem \ref{thm:paw}:

\begin{theorem}\label{thm:seagull}
    Let $G$ be a (clock, diamond)-free graph with no star cutset. Suppose that $G$ contains a seagull with $a, u, v \in V(G)$ as in the definition of a seagull. Suppose further that $a$ is a claw center in $G$. Then there is a vertex $b\in V(G)$ and a clique $K \subseteq N[b]$ such that $\{v\} \cup K$ separates $\{a\}$ from $\{u\}$, and $a$ is non-adjacent to $b$. 
\end{theorem}
\begin{proof}
    Suppose not. 
    
    \sta{\label{st:av}The set $N(a) \cap N(v)$ is empty.}
    
    Suppose that there is a vertex $a' \in N(a) \cap N(v)$. Then, the vertices $a, a', u, v$ form a paw (using that $G$ is diamond-free), and so \ref{thm:seagull} follows from Theorem \ref{thm:paw}, a contradiction. This proves \eqref{st:av}. 

    \medskip
    
    Let $\mathcal{Q}$ be the set of all three-path configuration of the form guaranteed by Lemma \ref{lem:seagull}; that is, every $Q \in \mathcal{Q}$ is a three-path configuration containing $a, u, v$, and $a$ is a claw center in $Q$. By Lemma \ref{lem:seagull}, the set $\mathcal{Q}$ is non-empty. Let $Q \in \mathcal{Q}$. Since $Q$ contains a claw center, it follows that $Q$ is a pyramid with apex $a$, or a theta with end $a$. Let $P_1(Q), P_2(Q), P_3(Q)$ be the paths of $Q$, labelled in such a way that $v \in P_1(Q)$. We define $b(Q)$ to be the end of $P_3(Q)$ which is not equal to $a$, and we define $l(Q) = |P_1(Q) \cup b(Q)|$. See Figure \ref{fig:seagull}. 

    Now, let $Q \in \mathcal{Q}$ be chosen with $l(Q)$ minimum. We claim that $b = b(Q)$ is the desired vertex. Note that $a$ and $b$ are not adjacent (if $Q$ is a theta, this is true from the definition of a theta; if $Q$ is a pyramid, it follows from Lemma \ref{lem:short}). As in Theorem \ref{thm:paw}, it is sufficient to prove:
    
    \sta{\label{st:seagullmain}The set $\{v\} \cup (N[b] \setminus \{u\})$ separates $\{u\}$ from $\{a\}$ in $G$.}

    Assuming \eqref{st:seagullmain} to be true, by applying Lemma \ref{startoclique} to $G \setminus \{v\}$ and the star cutset $N[b] \setminus \{u\}$, we obtain the desired clique $K$. 

    It remains to prove \eqref{st:seagullmain}; so we suppose for a contradiction that \eqref{st:seagullmain} does not hold. We first show that $u \neq b$. In the case that $Q$ is a pyramid, this is immediate; when $Q$ is a theta, it follows from the fact that $v$ does not have two non-adjacent neighbours (namely $a$ and $u=b$) in the hole $P_2(Q) \cup P_3(Q)$. Therefore, $u \neq b$. 

    Next, let us define some notation. We denote the neighbour of $a$ in $P_2(Q)$ as $x$, and the neighbour of $a$ in $P_3(Q)$ as $y$. Moreover, let us write $b_1$ for the end of $P_1(Q)$ not equal to $a$, and for $i \in \{2, 3\}$, let us write $b_i$ for the unique neighbour of $b$ in $P_i(Q)$; see Figure \ref{fig:seagull}. 
    
    Let $X = \{v\} \cup (N[b] \setminus \{u\})$. Let $Y = P_1(Q) \setminus \{v, b\}$; it follows that $u \in Y$ as $u \neq b$. Let $Z = (P_2(Q) \cup P_3(Q)) \setminus N[b]$; it follows that $a \in Z$. Since we assumed that \eqref{st:seagullmain} does not hold, it follows that there is a path from $Y$ to $Z$ with interior disjoint from $X$; let $R$ be a shortest such path. It follows that one end of $R$ is in $Y$, and the other is in $Z$, and $R^*$ is disjoint from $X \cup Y \cup Z$. 

    \begin{figure}[t]
        \centering
        \includegraphics[width=0.8\textwidth]{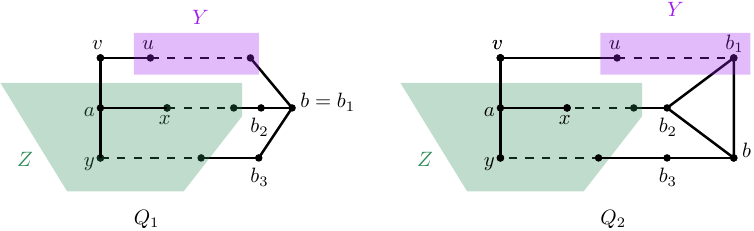}
        \caption{Name conventions for vertices in the proof of Theorem \ref{thm:seagull} in the case when $Q = Q_i$ is a theta ($i=1$) and when $Q = Q_i$ is a pyramid ($i = 2$). \pictured}
        \label{fig:seagull}
    \end{figure}

    \sta{\label{st:aboutr}The set $R^*$ is non-empty, and $N[b] \cap R^* = \emptyset.$}
    Since $Y$ is anticomplete to $Z$, it follows that $R^*$ is non-empty.
     Moreover, since $N[b] \subseteq X \cup Y$, it follows that $N[b] \cap R^* = \emptyset$. This proves \eqref{st:aboutr}. 

     \medskip
     
     Let $r_1, \dots, r_t$ denote the vertices of $R^*$ in order, such that $r_1$ has a neighbour in $Y$ and $r_t$ has a neighbour in $Z$. By \eqref{st:aboutr}, the only vertices of $Q$ that may have a neighbour in $R^* \setminus \{r_1, r_t\}$ are $v, b_2, b_3$. 

    By considering the holes $P_i(Q) \cup P_j(Q)$ for distinct $i, j \in \{1, 2, 3\}$, we conclude that $N(x) \cap Q$ is a clique for all $x \in V(G) \setminus Q$. 

    \sta{\label{st:i}There exists $j \in \{2, 3\}$ such that $N(r_t) \cap Q \subseteq P_j(Q)$.}

    Suppose not. Then, since $r_t$ is non-adjacent to $b$ by \eqref{st:aboutr}, and since $N(r_t) \cap Q$ is a clique, it follows that $r_t$ is adjacent to $v$ and $a$. This contradicts \eqref{st:av} and proves \eqref{st:i}. 

    \medskip

    In the remainder of this proof, let us fix $j$ as in \eqref{st:i} and let $i \in \{2, 3\} \setminus \{j\}.$ Since $N(r_t) \cap Q$ is a clique and $b$ is non-adjacent to $r_t$, it follows that $N(r_t) \cap P_i(Q) \subseteq \{a\}$. 

    \sta{\label{st:nbrs2}The vertex $b_i$ has a neighbour in $R^*$. Furthermore, if $Q$ is a pyramid and $b_3$ has a neighbour in $R^*$, then, traversing $R^*$ from $r_1$ to $r_t$, the first neighbour of $b_3$  appears at the same time or after the first neighbour of $b_2$.}

    The proof of \eqref{st:nbrs2} is similar to the proof of \eqref{st:bi}. Suppose that \eqref{st:nbrs2} does not hold. Let $R'$ be a path from $r_t$ to $a$ with interior in $P_j(Q) \setminus N[b]$. Let us define a hole $H$ and paths $T, T'$ as follows:  
    \begin{itemize}
        \item If $b_i$ has no neighbour in $R^*$, then $H = r_t$-$R'$-$a$-$P_i(Q)$-$b_1$-$P_1(Q)$-$r_1$-$R^*$-$r_t$ and $T= r_1$-$R^*$-$r_t$-$R'$ and $T' = P_i(Q)$.
        \item Otherwise, $Q$ is a pyramid and $b_3$ has a neighbour in $R^*$; let $\hat{R}$ be a path from $r_1$ to $b_3$ with interior in $R^*$. From our assumption, it follows that $\hat{R}$ contains no neighbour of $b_2$. We set $H = r_1$-$\hat{R}$-$b_3$-$P_3(Q)$-$a$-$P_2(Q)$-$b_2$-$b_1$-$P_1(Q)$-$r_1$ and $T = r_1$-$\hat{R}$-$b_3$-$P_3(Q)$-$a$ and $T' = P_2(Q)$. 
    \end{itemize}
    In either case, we have $T \cap T' = \{a\}$. 
    
    Suppose first that $v$ has a neighbour $q$ in $R^* \cap H$. Then, since $G$ does not contain a clock and $v$ is adjacent to $a \in H$, it follows that $qa \in E(G)$. But now $q \in N(a) \cap N(v)$, contrary to \eqref{st:av}. Therefore, $v$ has no neighbour in $R^* \cap H$. 
    
    Now we consider $N(r_1) \cap Q$, which is a clique. We would like to show that $N(r_1) \cap Q \subseteq P_1(Q)$. If $r_1$ has a neighbour in $P_1(Q)^*$, then $N(r_1) \cap Q \subseteq P_1(Q)$, as desired. Otherwise, $r_1$ is adjacent to $b_1$ as well as at least one of $b, b_2$ and $Q$ is a pyramid. Since $G$ is diamond-free, it follows that $r_1$ is adjacent to both $b$ and $b_2$. But this contradicts \eqref{st:aboutr}. We conclude that $N(r_1) \cap Q \subseteq P_1(Q)$. 
    
    If $N(r_1) \cap Q = \{b'\}$ is a single vertex, then there is a theta $Q'$ with paths $P_1(Q') = b'$-$P_1(Q)$-$a$ (which contains $u$ since $r_1$ has a neighbour in $Z$), $P_2(Q') = b'$-$P_1(Q)$-$b_1$-$T'$-$a$, and $P_3(Q') = b'$-$r_1$-$T$-$a$. It follows that $Q' \in \mathcal{Q}$. Since $P_1(Q') \cup \{b(Q')\} = P_1(Q') \subsetneq P_1(Q) \cup \{b\}$ since $b \not\in P_1(Q')$, we conclude that $l(Q') < l(Q)$, contrary to the choice of $Q$. This implies that $N(r_1) \cap Q = \{b', c\}$, where $b'c \in E(Q)$, and we may assume that $P_1(Q)$ traverses $v, c, b'$ in this order. We note that $b' \neq b$ by \eqref{st:rdisjoint}. Now there is a pyramid $Q'$ in $G$ with paths $P_1(Q') = c$-$P_1(Q)$-$a$ (which contains $u$, since $Q'$ is not a short pyramid by Lemma \ref{lem:short}), $P_2(Q') = b'$-$P_1(Q)$-$T'$, and $P_3(Q') = T$. It follows that $Q' \in \mathcal{Q}$.  Again, we have that $P_1(Q') \cup \{b(Q')\} = P_1(Q') \cup \{b'\} \subsetneq P_1(Q) \cup \{b\}$, contradicting the choice of $Q$. This proves \eqref{st:nbrs2}. 

    \medskip

Let $R''$ be a shortest path from $r_t$ to $b_j$ with interior in $P_j(Q) \setminus N[b])$. Let $P = r_1$-$R$-$r_t$-$(R''\setminus \{b_j\})$; so in particular, $N(b) \cap P = \emptyset$. Let $P'$ be the shortest subpath of $P$ containing $r_1$ as well as a neighbour of $b_i$ and a neighbour of $b_j$. Since $R''$ contains a neighbour of $b_j$ and $R^*$ contains a neighbour of $b_i$ by \eqref{st:bi}, the path $P'$ is well-defined. Let $p$ be the end of $P'$ not equal to $r_1$; and let $k \in \{2, 3\}$ be maximum such that $b_k$ is adjacent to $p$ and $b_k$ has no neighbour in $P' \setminus \{p\}$ (such $k$ exists by the choice of $P'$, as otherwise $P' \setminus \{p\}$ would be a better choice that $P'$). Let $k' \in \{2, 3\} \setminus \{k\}$. 

Suppose first that either $Q$ is a theta, or $Q$ is a pyramid and $k=3$. Then, there is a hole $H'$ in $G$, defined as $H' = r_1$-$P_1(Q)$-$b$-$b_k$-$p$-$P'$-$r_1$. The vertex $b_{k'}$ has two non-adjacent neighbours in $H'$, namely $b$ and a neighbour in $P'$. Since $G$ is clock-free, this is a contradiction. 

It follows that $Q$ is a pyramid and $k = 2$. Since at least one of $b_2, b_3$ has a neighbour in $R^*$ by \eqref{st:nbrs2}, it follows from the choice of $k$ that $b_3$ has a neighbour in $R^*$. But now \eqref{st:nbrs2} implies that the first neighbour of $b_2$ along $R^*$, traversed from $r_1$ to $r_t$, appears at the same time or before the first neighbour of $b_3$, which implies that $k = 3$, a contradiction. This concludes the proof. 
\end{proof}

\section{Central bags}

In the previous section, we showed that paws and certain seagulls lead to cutsets in (clock, diamond)-free graphs. In this section, we will set up the ``central bag method,'' which, under certain circumstances, allows us to decompose a graph along several cutsets simultaneously, obtaining a much simplified graph -- the \emph{central bag} -- as a result. Then, using the structure of the central bag, we show that it has small treewidth. Finally, we ``lift'' a certificate of small treewidth for the central bag to a certificate for the original graph by carefully reversing the decompositions. 

As a first step, let us describe the ``certificate'' of small treewidth. Rather than working with a tree decomposition, we will work with balanced separators: Let $G$ be a graph. A \emph{weight function} on $G$ is a function $w: V(G) \rightarrow [0, 1]$ such that $\sum_{v \in V(G)} w(v) = 1$. For $X \subseteq V(G)$, we write $w(X)$ for $\sum_{x \in X} w(x)$. 

Let $c \in [0,1]$ and let $G$ be a graph with weight function $w$. A set $X \subseteq V(G)$ is a \emph{$(w,c)$-balanced separator} for $G$ if for every component $D$ of $G \setminus X$, we have $w(D) \leq c$. The following two lemmas show that treewidth is closely related to the existence of balanced separators: 

\begin{lemma}[Harvey and Wood \cite{harvey2017parameters}; stated in this form in \cite{abrishami2021induced}]\label{lemma:bs-to-tw}
Let $G$ be a graph, let $c \in [\frac{1}{2}, 1)$, and let $k$ be a positive integer. If $G$ has a $(w, c)$-balanced separator of size at most $k$ for every weight function $w$ on $G$, then $\tw(G) \leq \frac{1}{1-c}k$. 
\end{lemma}

\begin{lemma}[Cygan, Fomin, Kowalik, Lokshtanov, Marx, Pilipczuk, Pilipczuk, and Saurabh \cite{cygan2015parameterized}]
\label{lemma:tw-to-weighted-separator}
Let $G$ be a graph and let $k$ be a positive integer. If $\tw(G) \leq k$, then $G$ has a $(w, c)$-balanced separator of size at most $k+1$ for every $c \in [\frac{1}{2}, 1)$ and for every weight function $w$ on $G$.
\end{lemma}

Therefore, to prove Theorem \ref{thm:mainnew}, it suffices to show that for every $t \in \mathbb{N}$, there exists a $k = k(t)$ such that for every $t$-clean (clock, diamond)-free graph $G$ and every weight function $w$ on $G$, there is a \whalf-balanced separator of size at most $k$ in $G$. In particular, we can fix a weight function $w$ and assume (for a contradiction) that for every set $X$ of size at most $k$, at least one (and therefore exactly one) component of $G \setminus X$ has weight more than $\frac{1}{2}$. 

Let us now turn to the cutsets we use to create the central bag. Let $G$ be a graph, and let $w$ be a weight function on $G$. For a set $X$ which is not a \whalf-balanced separator of $G$, we define its \emph{canonical separation} $S_{w, G}(X) = (A, C, B)$ where $B$ is the unique component $D$ of $G \setminus X$ with $w(D) > \frac{1}{2}$, $C = X$, and $A = V(G) \setminus (B \cup C)$. We note that $A$ is anticomplete to $B$. 

Let $G$ be a graph, and let $w$ be a weight function on $G$. Let $\mathcal{X}$ be a set of subsets of $V(G)$, none of which is a \whalf-balanced separator of $G$. Then, we define the \emph{central bag} of $G$ and $w$ with respect to $\mathcal{X}$ as 
$$\beta(G, w, \mathcal{X}) = \bigcap_{X \in \mathcal{X} \textnormal{ with } S_{w, G}(X) = (A, C, B)} (B \cup C).$$
In other words, we delete all the ``$A$-sides'' of canonical separations corresponding to sets $X \in \mathcal{X}$. 

We would like to say that each component of $G \setminus \beta(G, w, \mathcal{X})$ is contained in $A$ for some $S_{w, G}(X) = (A, C, B)$ with $X \in \mathcal{X}$. To arrange this, we define the following. Let us say that two separations $(A, C, B)$ are $(A', C', B')$ of a graph $G$ are \emph{loosely non-crossing} if there is no path in $G$ with one end in $A \cap C'$, the other end in $A' \cap C$, and with interior in $A \cap A'$. We observe:

\begin{lemma} \label{lem:noncross}
    Let $G$ be a graph, and let $w$ be a weight function for $G$.  Let $\mathcal{X}$ be a set of subsets of $V(G)$, none of which is a \whalf-balanced separator of $G$. Suppose that for all $X, X' \in \mathcal{X}$, the canonical separations $S_{w, G}(X)$ and $S_{w, G}(X')$ are loosely non-crossing. Then, for every component $D$ of $G \setminus \beta(G, w, \mathcal{X})$, there is an $X \in \mathcal{X}$ such that $D \subseteq A$, where $(A, C, B) = S_{w, G}(X)$. 
\end{lemma}
\begin{proof}
    From the definition of $\beta(G, w, \mathcal{X})$, it follows that there is an $X \in \mathcal{X}$ such that $D \cap A \neq \emptyset$, where $(A, C, B) = S_{w, G}(X)$. Among all such $X$, let us choose $X$ and a component $D_X$ of $D \cap A$ with $|D_X|$ as large as possible; and fix $(A, C, B) = S_{w, G}(X)$. 

    If $D_X = D$, then the lemma holds; so we may assume that there is a vertex $v' \in D \cap N(D_X)$ (as $D$ is connected). Again from the definition of $\beta(G, w, \mathcal{X})$, it follows that there is an $X' \in \mathcal{X}$ such that $v' \in A'$, where $(A', C', B') = S_{w, G}(X')$. 

    Let $D_{X'}$ be the component of $(D_X \cup \{v'\}) \cap A'$ containing $v'$ (which exists, as $v' \in A'$). From the choice of $X$ and $D_X$, it follows that $D_{X'}$ does not contain all vertices of $D_X$. Since $D_X \cup \{v'\}$ is connected and $D_{X'}$ is a proper subset of $D_X \cup \{v'\}$, it follows that there is a vertex $v \in D_X \setminus D_{X'}$ with a neighbour in $D_{X'}$. From the choice of $D_{X'}$, it follows that $v \not\in A'$. 

    Now, let $P$ be a path from $v$ to $v'$ with interior in $D_{X'}$. Then: 
    \begin{itemize}
        \item $P \setminus \{v\} \subseteq D_{X'} \subseteq A'$ and $P \setminus \{v'\} \subseteq D_{X'} \setminus \{v'\} \subseteq D_X \subseteq A$; 
        \item $v$ is not in $A'$, but $v$ has a neighbour in $A'$ (along $P$), and so $v \in C'$; 
        \item $v'$ is not in $A$, but $v'$ has a neighbour in $A$ (along $P$), and so $v \in C$. 
    \end{itemize}
    This is shows that $(A, C, B)$ and $(A', C', B')$ are not loosely non-crossing, contrary to our assumption. 
\end{proof} 

Given two sets $X, X'$ which are not \whalf-balanced separators of $G$, we say that $X$ is a \emph{$(w, G)$-shield} for $X'$ (a notion also used in, for example, \cite{abrishami2021induced, abrishami2021submodular, abrishami2022induced2}) if, writing $S_{w, G}(X) = (A, C, B)$ and $S_{w, G}(X') = (A', C', B')$, we have that one of the following holds: 
\begin{itemize}
    \item $B \cup C \subsetneq B' \cup C'$; or 
    \item $B \cup C = B' \cup C'$ and $B' \subsetneq B$.
\end{itemize}
From this definition, it is immediate that for every set $\mathcal{X}$ of subsets of $V(G)$ none of which is a \whalf-balanced separator of $G$, being a $(w, G)$-shield is a partial order on $\mathcal{X}$. We now consider the set of all ``minimal'' elements of $\mathcal{X}$ with respect to this order. Explicitly, we define 
$$\core_{w, G}(\mathcal{X}) = \{X' \in \mathcal{X} : \textnormal{there is no } X \in \mathcal{X} \textnormal{ which is a } (w, G) \textnormal{-shield for } X'\}.$$
It follows that for every $X' \in \mathcal{X} \setminus \core_{w, G}(\mathcal{X})$, there is an $X \in \core_{w, G}(\mathcal{X})$ which is a shield for $X'$. 

The idea is that if $X$ is a shield for $X'$, then $X$ is ``strictly more useful'' than $X'$ (given that the sets $B$ and $B'$ are the large components, and we would like the large components to be as small as possible); so it suffices to consider cutsets $X \in \core_{w, G}(\mathcal{X})$. 

So far, this description is common to applications of the central bag method; the main difference is in the types of sets $X$ we use for canonical separations, which we describe in the next section. 

\section{2-clique cutsets}

Let $G$ be a diamond-free graph. For a clique $K \subseteq V(G)$ and a set $A \subseteq V(G)$, let us define $c_A(K)$ as follows: 
\begin{itemize}
    \item If $|K| \leq 1$, then $c_A(K) = K$.
    \item Otherwise, let $x, y \in K$ be distinct. Then $c_A(K) =     K \cup (N(x) \cap N(y) \cap A)$.
\end{itemize}
Note that, by Lemma \ref{lem:diamond}, the set $c_A(K)$ is a (well-defined) clique. 

Let $G$ be a diamond-free graph and let $w$ be a weight function on $G$. Let $K_1, K_2$ be cliques in $G$, and suppose that $X = K_1 \cup K_2$ is not a \whalf-balanced separator of $G$. Let $(A, C, B) = S_{w, G}(X)$. Then, we define $\closure(K_1, K_2) = c_{A \cup C}(K_1 \cap N(B)) \cup c_{A \cup C}(K_2 \cap N(B))$. We observe that:
\begin{itemize}
    \item The set $\closure(K_1, K_2)$ is the union of at most two cliques $K_1', K_2'$ where $\closure(K_1', K_2') = \closure(K_1, K_2)$.  
    \item $B$ is a component of $G \setminus \closure(K_1, K_2)$, and therefore, $\closure(K_1, K_2)$ is not a \whalf-balanced separator of $G$.
    \item Writing $(A', C', B') = S_{w, G}(\closure(K_1, K_2))$, we have $B' = B$. 
\end{itemize}

For a graph $G$, we denote by $\omega(G)$ the size of the largest clique in $G$. For a diamond-free graph $G$ and a weight function $w$ such that $G$ has no \whalf-balanced separator of size at most $4\omega(G)$, let us define $$\mathcal{X}(G) = \{\closure(K_1, K_2) : K_1, K_2 \textnormal{ are cliques in } G\}.$$

\begin{theorem}\label{thm:noncrossing}
    Let $G$ be a (clock, diamond)-free graph, and let $w$ be a weight function on $G$. Suppose that $G$ contains no \whalf-balanced separator of size at most $4\omega(G)$ and that $G$ has no  star cutset. Let $X, X' \in \core_{w, G}(\mathcal{X}(G))$. Then $S_{w, G}(X)$ and $S_{w, G}(X')$ are loosely non-crossing. 
\end{theorem}
\begin{proof}
Suppose not. Let $(A, C, B) = S_{w, G}(X)$ and $(A', C', B') = S_{w, G}(X')$. Let $K_1, K_2$ be cliques such that $C = K_1 \cup K_2 = \closure(K_1, K_2)$, and let $K_1', K_2'$ be cliques such that $C' = K_1' \cup K_2' = \closure(K_1', K_2')$. 

Since $(A, C, B)$ and $(A', C', B')$ are not loosely non-crossing, and by symmetry, we may assume that there is a path $P$ from $r' \in A \cap K_1'$ to $r \in A' \cap K_1$ with interior in $A \cap A'$. 

Since $G$ has no \whalf-balanced separator of size at most $4 \omega(G)$, it follows that there is a component $B^*$ of $G \setminus (C \cup C')$ with $w(B^*) > \frac{1}{2}$. 

Since $w(B) > \frac{1}{2}$ and $w(B') > \frac{1}{2}$ from the definition of $S_{w,G}(\cdot)$, it follows that $B^* \cap B' \neq \emptyset$ and $B^* \cap B \neq \emptyset$, and therefore, since $B^*$ is disjoint from $C \cup C'$, it follows that $B^* \subseteq B \cap B'$. 

\sta{\label{st:nbrsbig}We have that $N(B^*) \subseteq K_2 \cup K_2' \cup (K_1 \cap K_1')$.}
Suppose not. Since $B^* \subseteq B \cap B'$, it follows that $N(B^*) \subseteq (B \cap C') \cup (B' \cap C) \cup (C \cap C')$. Since $r \in A' \cap K_1$, it follows that $B' \cap C \subseteq K_2$; similarly, $B \cap C' \subseteq K_2'$. Finally, $(C \cap C') \setminus (K_2 \cup K_2') \subseteq K_1 \cap K_1'$. This implies \eqref{st:nbrsbig}. 

\sta{\label{st:k1k2} The set $K_1 \cap K_1'$ is non-empty.}
Suppose not. By \eqref{st:nbrsbig}, it follows that $N(B^*) \subseteq K_2 \cup K_2'$. Let $K_3, K_4$ be cliques such that $C^* = K_3 \cup K_4 = \closure(K_3, K_4) = \closure(K_2 \cap N(B^*), K_2'\cap N(B^*))$. Letting $A^* = V(G) \setminus (C^* \cup B^*)$, we have that $S_{w, G}(C^*) = (A^*, C^*, B^*)$. We also have $C^* \in \mathcal{X}(G)$. We notice that: 
\begin{itemize}
    \item $C^* \cup B^* \subseteq B \cup C$: Suppose not. Clearly, $B^* \cup N(B^*) \subseteq B \cup C$. Therefore, it follows that there is a vertex $x$ in $K_3 \cup K_4$ which is not in $B \cup C$. From the definition of $K_3$ and $K_4$, if follows that either $x$ has at least two neighbours in $K_2 \cap N(B^*)$ (but then $x \in K_2 \subseteq C$) or $x$ has at least two neighbours in $K_2' \cap N(B^*)$ (but then $x \in K_2'$). We conclude that $x \in K_2' \cap A$. Since $r' \in A \cap K_1'$, and since $K_1', K_2'$ are cliques, it follows that $C' \subseteq A \cup C$. Therefore, $C' \cap B = \emptyset$. Since $B$ is connected and $B \cap B' \neq \emptyset$, we conclude that $B \subseteq B'$. Now, since $r \in A' \cap K_1$ has no neighbour in $B$, it follows that $r \in c_{A \cup C}(N(B) \cap K_1)$, and so $r$ has at least two neighbours in $N(B) \cap K_1 \subseteq C \cap C'$. Since $K_1 \cap K_1 = \emptyset$, it follows that $N(B) \cap K_1 \subseteq K_2'$. But then $r \in c_{A' \cup C'}(K_2')$, and so $r \in C'$, a contradiction. 
    \item The vertex $r$ is in $A^*$: Suppose not. Since $r \not\in B^* \cup N(B^*)$ (as $B^* \subseteq B$ and $N(B^*) \subseteq B \cup C$), it follows that $r \in C^*$ and $r$ has at least two neighbours in one of $K_2 \cap N(B^*)$ and $K_2' \cap N(B^*)$. Since $r \in A'$, it follows that $r \not\in c_{A' \cup C'}(K_2')$, and so $r$ does not have two or more neighbours in $K_2'$. It follows that $r$ has two or more neighbours in $K_2 \cap N(B^*)$. Then, as $G$ is diamond-free, it follows that $r$ is adjacent to every vertex in $K_2$. Since $r$ is also adjacent to every vertex in $K_1$ (as $r \in K_1$ and $K_1$ is a clique), it follows that $C$ is a star cutset (separating $r'$ from $B$). This contradicts the assumption that $G$ has no star cutset. 
\end{itemize}
Putting the above two items together, it follows that $C^* \cup B^* \subsetneq B \cup C$, and so $C^*$ is a $(w, G)$-shield for $C = X$, a contradiction to the assumption that $X \in \core_{w, G}(\mathcal{X}(G))$. This proves \eqref{st:k1k2}.

\medskip 

Since $r \in A'$, it follows that $r$ has at most one neighbour in $K_1'$. But $r$ is adjacent to every vertex in $K_1 \cap K_1'$ as $r \in K_1 \setminus K_1'$; and so by \eqref{st:k1k2}, it follows that $|K_1 \cap K_1'| = 1.$ Let $v \in K_1 \cap K_1'$.

\begin{figure}[t]
    \centering
    \includegraphics[width=0.4\textwidth]{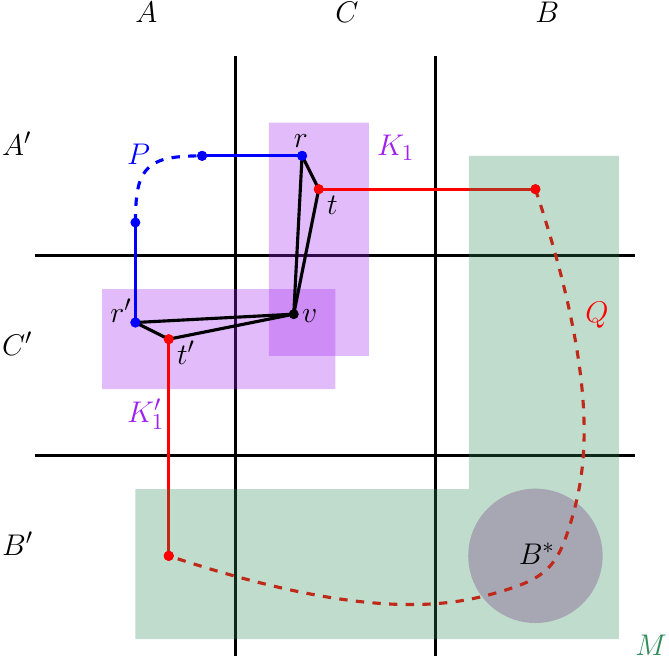}
    \caption{The proof of Theorem \ref{thm:noncrossing}. \pictured}
    \label{fig:crossing}
\end{figure}

Next, let us define $M = B \cup B'$. Since $B$ and $B'$ are connected, and have non-empty intersection $B^*$, it follows that $M$ is connected. Moreover, since $K_1 \cup K_1' \subseteq (C \cup A) \cap (C' \cup A')$, it follows that $K_1 \cup K_1'$ is disjoint from $M$. 

We define a vertex $t$ as follows: Since $r \in c_{A \cup C}(K_1 \cap N(B))$, it follows that either $r$ has a neighbour in $B$, and we let $t = r$; or the set $K_1 \cap N(B)$ contains at least two vertices, and we choose $t \in (K_1 \cap N(B)) \setminus \{v\}$. Analogously, we define a vertex $t'$ as follows: Since $r' \in c_{A' \cup C'}(K_1' \cap N(B'))$, it follows that either $r'$ has a neighbour in $B'$, and we let $t' = r'$; or the set $K_1' \cap N(B')$ contains at least two vertices, and we choose $t' \in (K_1' \cap N(B')) \setminus \{v\}$.

We note that if $r \neq t$, then $r$ has no neighbour in $M$ (as $r$ has no neighbour in $B$ from the choice of $t$, and $r$ has no neighbour in $B'$ since $r \in A'$); an analogous statement holds for $r'$. Let $Q$ be a path from $t$ to $t'$ with interior in $M$. 

Since both $t$ and $t'$ have a neighbour in $P$, we choose $P'$ to be a path from $t$ to $t'$ with interior contained in $P$. 

We claim that $H = t$-$P'$-$t'$-$Q$-$t$ is a hole: No vertex in $P^*$ has a neighbour in $M$, and if $r$ is in the interior of $P'$, then $r \neq t$ and so $r$ has no neighbour in $M$; and similarly for $r'$. Therefore, there are no edges from $P'^*$ to $Q^*$; and both $P'$ and $Q$ are induced paths. 

Moreover, we claim that $v \not\in H$: From the choices of $t$ and $t'$, we have that $v \neq t, t'$; and since both $P$ and $Q$ are disjoint from $C \cap C'$, it follows that $P'$ and $Q$ are disjoint from $\{v\}$. 

Since $t \in K_1$ and $t' \in K_1'$, it follows that $t, t'$ are two distinct neighbours of $v$ in $H$. Moreover, since $t \not\in c_{A' \cup C'}(K_1')$, it follows that $t$ has at most one neighbour in $K_1'$, namely $v$; so $t$ is non-adjacent to $t'$. This implies that $\{v\} \cup H$ is a clock, which is a contradiction and completes the proof.
\end{proof}

\section{Inside the central bag} \label{sec:inside}

Throughout this section, we make the following assumption: 
\begin{assumptions}\label{as:centralbag}
Let $G$ be a (clock, diamond)-free graph with no star cutset and let $w$ be a weight function on $G$. Suppose that $G$ contains no \whalf-balanced separator of size at most $4\omega(G)$. Let $\mathcal{X}(G)$ be as in the previous section, and let $\beta = \beta(G, w, \core_{w, G}(\mathcal{X}(G)))$.
\end{assumptions}
Following the central bag method, we now have two goals: 
\begin{itemize}
    \item Show that $\beta$ has a small balanced separator; and
    \item ``Lift'' this separator to $G$. 
\end{itemize}
In pursuit of the second goal, it is helpful to extend the central bag $\beta$ to incorporate certain bits from each component of $G \setminus \beta$. This will help with lifting, as it prevents these components from creating arbitrary connections that we do not ``see'' in $\beta$. Instead, we keep certain ``marker paths'' to record those connections. To make our lives easier, we would like to choose marker paths that are as simple as possible; the following lemma helps with this. 

\begin{lemma}\label{lem:marker}
    We assume that Assumption \ref{as:centralbag} holds. Let $X \in \mathcal{X}(G)$. Let $(A, C, B) = S_{w, G}(X)$ and let $D$ be a component of $A$. Let $K_1, K_2$ be cliques such that $X = K_1 \cup K_2 = \closure(K_1, K_2).$ Then there is a path $P$ with ends $x, y \in X$ and interior in $D$ such that the following hold: 
    \begin{itemize}
        \item The path $P$ has at least three vertices (in other words, $x$ and $y$ are not adjacent); and
        \item Every vertex in $P^*$ has degree 2 in the graph $G[P \cup X]$ (and is therefore not contained in a triangle in $G[P \cup X]$). 
    \end{itemize}
\end{lemma}
\begin{proof}
    Since $G$ has no star cutset, we find that $K_1 \cap K_2 = \emptyset$. We start with the following: 

    \sta{\label{st:betweenliques}No vertex in $K_1$ has two or more neighbours in $K_2$, and vice versa.}
    Suppose for a contradiction that $v \in K_1$ has at least two neighbours in $K_2$. Then, since $G$ is diamond-free, it follows that $v$ is adjacent to every vertex in $K_2$. But then $X \subseteq N[v]$, and so $G$ has a star cutset, a contradiction. This proves \eqref{st:betweenliques}. 

    \medskip

    Since $G$ has no star cutset, and hence no clique cutset, it follows that $N(D)$ is not a clique. Therefore, there is a path with interior in $D$ and non-adjacent ends in $N(D) \subseteq X$. Let $P$ be a shortest such path, and let $x$ and $y$ be its ends. We may assume, by symmetry, that $x \in K_1$ and $y \in K_2$. 

    \sta{\label{st:marker} Let $v \in X \setminus \{x, y\}$ such that $v$ has a neighbour in $P^*$. Then $v$ is adjacent to both $x$ and $y$.}

Suppose not. We may assume that  $v \in K_1$, and so $v$ is adjacent to $x$ and non-adjacent to $y$. From the choice of $P$, since the path from $v$ to $y$ with interior in $P^*$ is not a better choice for $P$, it follows that $v$ is adjacent to the neighbour $x^*$ of $x$ in $P$. But now $x^* \in \closure(K_1,K_2)$, a contradiction. This proves \eqref{st:marker}. 

    \medskip

    Let us pick $x'$ as follows: If $x$ has a neighbour in $B$, then $x' = x$. Otherwise, at least two vertices in $K_1 \setminus \{x\}$ have a neighbour in $B$. Note that $y$ has at most one neighbour in $K_1$ by \eqref{st:betweenliques}. We choose $x' \in K_1 \setminus N(y)$ such that $x'$ has a neighbour in $B$. 

    Let us now pick $y'$. If $y$ has a neighbour in $B$, then we let $y' = y$. Otherwise, at least two vertices in $K_2 \setminus \{y\}$ have a neighbour in $B$. Since $x$ has at most one neighbour in $K_2$, we pick $y'$ to be a vertex in $K_2 \setminus N(x)$ with a neighbour in $B$. 

    \begin{figure}[ht]
        \centering
        \includegraphics[width=0.4\textwidth]{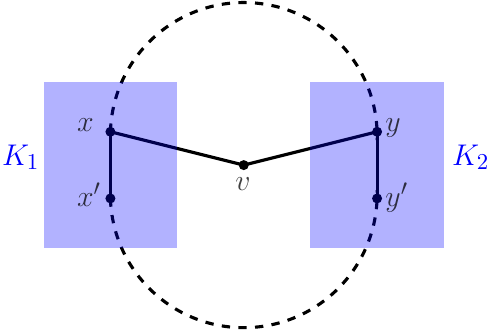}
        \caption{Proof of Lemma \ref{lem:marker}. A clock in the case that some vertex in $X$ is adjacent to both $x$ and $y$. \pictured}
        \label{fig:marker}
    \end{figure}

    From our choices described above, it follows that: 
    \begin{itemize}
        \item the vertices $x, x'$ are in $K_1$ (and possibly equal); 
        \item the vertices $y, y'$ are in $K_2$ (and possibly equal); 
        \item $y$ is anticomplete to $\{x, x'\}$ and $x$ is anticomplete to $\{y, y'\}$; 
        \item if $x \neq x'$, then $x'$ has no neighbour in $P^*$ (by \eqref{st:marker}, since $x'$ is non-adjacent to $y$); and
        \item If $y \neq y'$, then $y'$ has no neighbour in $P^*$ (by \eqref{st:marker}, since $y'$ is non-adjacent to $x$). 
    \end{itemize}
    Now, let $Q$ be a path from $x'$ to $y'$ with interior in $B$; this exists by the choice of $x'$ and $y'$. Then, $H = x'$-$Q$-$y'$-$y$-$P$-$x$-$x'$ is a hole in $G$. If some vertex $v \in (K_1 \cup K_2) \setminus \{x, y\}$ has a neighbour in $P^*$, then by \eqref{st:marker}, $v$ has two non-adjacent neighbours in $H$, namely $x$ and $y$, a contradiction as $G$ is clock-free. It follows that no vertex in $(K_1 \cup K_2) \setminus \{x, y\}$ has a neighbour in $P^*$. Since $P$ is induced, each of $x$ and $y$ has exactly one neighbour in $P^*$. This completes the proof. 
\end{proof}

By Theorem \ref{thm:noncrossing}, the separations $S_{w, G}(X)$ and $S_{w, G}(X')$ are loosely non-crossing for all $X, X' \in \core_{w, G}(\mathcal{X}(G))$. By Lemma \ref{lem:noncross}, we may define, for each component $D$ of $G \setminus \beta$, a set $X(D) 
 \in\core_{w, G}(\mathcal{X}(G))$ such that, writing $(A, C, B) = S_{w, G}(X(D))$, we have that $D \subseteq A$. It follows that $D$ is a component of $A$. 

Given $X \in \core_{w, G}(\mathcal{X}(G))$, we write $D(X)$ for the union of all components $D$ of $G \setminus \beta$ with $X(D) = X$. It follows that 
$$\bigcup_{X \in \core_{w, G}(\mathcal{X}(G))} D(X) = G \setminus \beta,$$
where the union is a disjoint union. 

Now, for every $X \in \core_{w, G}(\mathcal{X}(G))$, let us define $P(X)$ as follows. If $D(X) = \emptyset$, then $P(X) = \emptyset$. Otherwise, let us pick a component $D$ of $D(X)$, and let $P(X)$ be a path with interior in $D$ and ends in $X$ as guaranteed by Lemma \ref{lem:marker}. We call $P(X)$ the \emph{marker path} for $X$. 

We define $$\beta^* = \beta \cup \bigcup_{X \in \core_{w, G}(\mathcal{X}(G))} P(X).$$  
From the choice of paths $P(X)$ as in Lemma \ref{lem:marker}, and since $D$ is a component of $G \setminus \beta$, it follows that every vertex in $\beta^* \setminus \beta$ has degree two in $\beta^*$ and is not contained in a triangle in $\beta^*$. 

In preparation for the next section, let us also define a weight function $w^*$ on $\beta^*$, as follows: 
\begin{itemize}
    \item For every $v \in \beta$, we let $w^*(v) = w(v)$. 
    \item For every $X \in \core_{w, G}(\mathcal{X}(G))$ with $D(X) \neq \emptyset$, we pick a vertex $a_X \in (P(X))^*$ arbitrarily; and we set $w^*(a_X) = w(D(X))$ and $w^*(v) = 0$ for all $v \in (P(X))^* \setminus \{a_X\}$. 
\end{itemize}
Then, for every vertex $v \in \beta$, its weight remains the same; for every component $D$ of $G \setminus \beta$, we move its total weight to $a_X$ where $X = X(D)$. From this, it is easy to see that $w^*$ is a weight function on $\beta^*$. 

The following results help us describe ways in which $\beta^*$ is structurally simpler than $G$. 
\begin{lemma}\label{lem:centralseagull}
    Assuming Assumption \ref{as:centralbag}, and with the definition of $\beta^*$ as above, the following holds. Suppose that $\beta^*$ contains a seagull with $a, u, v$ as in the definition of a seagull. Then one of the following holds: 
    \begin{itemize}
        \item At least one of $a$ or $u$ is in $G \setminus \beta$.
        \item $N_G(u) \setminus \{v\}$ is a clique anticomplete to $\{v\}$.
        \item $N_G(a) \setminus \{v\}$ is a clique anticomplete to $\{v\}$. 
    \end{itemize}
\end{lemma}
\begin{proof}

First we show

\sta{\label{st:wincondition}If $N_G(u) \setminus \{v\}$ is a clique, then the theorem holds.}

Suppose  $N_G(u) \setminus \{v\}$ is a clique $K$ and $v$ has a neighbour in
$K$. Then $N_G(u)$ is connected, and therefore a clique by Lemma \ref{lem:diamond}. 
Since $G$ contains a seagull, $G$ is not a complete graph. But now
$N_G(u)$ is a clique cutset in $G$, contradicting Assumption \ref{as:centralbag}. This proves \eqref{st:wincondition}. 

\medskip

By \eqref{st:wincondition}, we may assume that  $N_G(u) \setminus \{v\}$ is not a clique,
and similarly $N_G(a) \setminus \{v\}$ is not a clique. We may assume that the first outcome does not hold, and it follows that $a$ and $u$ are in $\beta$. 

    \sta{\label{st:72}There is a cutset in $G$ of the form $\{v\} \cup K$ where $K$ is a clique and which separates $\{a\}$ from $\{u\}$.}
    We consider two cases. Suppose first that $a$ is a claw center in $G$. Then, by Theorem \ref{thm:seagull}, there is a vertex $b$ non-adjacent to $a$ and a clique $K \subseteq N[b]$ in $G$ such that $\{v\} \cup K$ separates $\{a\}$ from $\{u\}$, and \eqref{st:72} holds. Therefore, we may assume that $a$ is not a claw center, and therefore (using Lemma \ref{lem:diamond}), we have that $N_G(a) = K_1 \cup K_2$, where $K_1$ and $K_2$ are cliques. We may assume that $v \in K_1$. If $K_1 = \{v\}$, then the third outcome holds. Therefore, we may assume that $K_1$ contains a vertex $a' \neq v$. Now, by Theorem \ref{thm:paw} applied to the paw induced by $\{u, v, a, a'\}$, we once again obtain the desired cutset. This proves \eqref{st:72}. 
    
    \medskip
    It follows that $X = \closure(K, \{v\}) \in \mathcal{X}(G)$, and so $\core_{w, G}(\mathcal{X}(G))$ contains either $X' = X$ or a $(w, G)$-shield $X'$ for $X$. Then, writing $(A, C, B) = S_{w, G}(X)$ and $(A', C', B') = S_{w, G}(X')$, it follows that $A \subseteq A'$ and $\beta \cap A = \emptyset$. 

    It follows that $a,u \not \in A$, and therefore $a \in X$ or $u \in X$. There is symmetry, and so we may assume that $u \in X$. Since $u \not \in K$, it follows that $u$ has at least two neighbours in $K$; consequently $u$ is adjacent to all vertices of $K$. Let $Q$ be the component of $G \setminus (\{v\} \cup K)$ containing $u$. Since $K \cup \{v\} \subseteq N[u]$, and since $G$ has no star cutset by Assumption \ref{as:centralbag}, it follows that $Q = \{u\}$, and so $N(u) = K \cup \{v\}$, and the third outcome holds.
\end{proof}

Let us say that a vertex $u$ in a graph $G$ is \emph{near-simplicial} in $G$ if there is a vertex $v \in V(G)$ such that $N_G(u) \setminus \{v\}$ is a clique. Then, the second and third outcomes of Lemma \ref{lem:centralseagull} imply that $u$ and $a$, respectively, are near-simplicial. 

\begin{lemma} \label{lem:centraldegree}
    Assuming Assumption \ref{as:centralbag}, and with the definition of $\beta^*$ as above, the following hold:

    \begin{enumerate}[(i)]
        \item For every $x \in \beta^*$, there is a set $Z_1 \subseteq N(x)$ such that $Z_1 = K \cup \{x'\}$ where $K$ is a clique, with the property that for every $v \in N_{\beta^*}(x) \setminus Z_1$, we have $\deg_{\beta^*}(v) \leq 2$. 
        \item For every $x$ in $\beta^*$ which is a claw center in $\beta^*$, there is a set $Z_2 \subseteq N(x)$ such that $Z_2 = K \cup \{x', x''\}$ where $K$ is a clique, with the property that for every neighbour $y$ of $x$ in $\beta^* \setminus Z_2$, there is no $X \in \core_{w, G}(\mathcal{X}(G))$ with $X = K_1 \cup K_2 = \closure(K_1, K_2)$ such that $y$ has a neighbour in a component of $D(X)$ and $x, y \in K_1$. 
    \end{enumerate}
    
\end{lemma}
\begin{proof}
    By our choice of $P(X)$ and by Lemma \ref{lem:marker}, it follows that all vertices in $\beta^* \setminus \beta$ have degree two in $\beta^*$ and are not contained in triangles in $\beta^*$. 
    
    Let $x \in \beta^*$. If $x \in \beta^* \setminus \beta$, then $\deg_{\beta^*}(x) = 2$, and both statements hold. So we may assume that $x \in \beta$. Moreover, if $x$ is near-simplicial in $\beta^*$, then again both (i) and (ii) hold; so we may assume that this is not the case. 
    
    \sta{\label{st:bowtie}The set $N_{\beta^*}(x)$, which is a disjoint union of cliques by Lemma \ref{lem:diamond}, contains at most one component of size more than one.}

    Suppose not; let $a, u \in N_{\beta^*}(x)$ be non-adjacent such that both $a$ and $u$ are in a clique of size at least two in $N_{\beta^*}(x)$; say $a'$ is a common neighbour of $x$ and $a$ in $\beta^*$. We apply Lemma \ref{lem:centralseagull} to the seagull with vertex set $\{a, x, u\}$. Then, since $a, u$ are each in a triangle in $\beta^*$ (with $x$ and a neighbour of $x$), it follows that $a, u \in \beta$, so the second or third outcome of Lemma \ref{lem:centralseagull} holds. By symmetry, we may assume that $N_G(a) \setminus \{x\}$ is a clique $K$ containing $a'$ and anticomplete to $\{x\}$. But $a'$ is adjacent to $x$, a contradiction; this proves \eqref{st:bowtie}. 

    \medskip
    Let $K$ be defined as follows. If $N_{\beta^*}(x)$ has no component of size more than one, then $K = \emptyset$. Otherwise, by \eqref{st:bowtie}, we define $K$ to be the unique component of $N_{\beta^*}(x)$ of size more than one; by Lemma \ref{lem:diamond}, we have that $K$ is a clique. 

    \sta{\label{st:twonbrsdeg}There is at most one vertex in $N_{\beta^*}(x) \setminus K$ with degree more than two in $\beta^*$.}
    Suppose that $a, u \in N_{\beta^*}(x) \setminus K$ both have degree more than two in $\beta^*$. Every neighbour of $x$ which is in $\beta^* \setminus \beta$ has degree two in $\beta^*$, and hence $a, u \in \beta$.  Applying Lemma \ref{lem:centralseagull} to the seagull with vertex set $\{a, x, u\}$, we conclude that the second or third outcome holds. By symmetry, we may assume that there is a clique $X$ such that $N_G(a) = X \cup \{v\}$ with $X$ anticomplete to $\{v\}$. Since $a$ has degree at least three in $\beta^*$, it follows that there are two distinct vertices $y, y' \in X \cap \beta^*$. Since both are in the triangle $\{a, y, y'\}$, it follows that $y, y' \in \beta$. But now, by applying Lemma \ref{lem:centralseagull} to the seagull with vertex set $\{x, a, y\}$, we conclude that one of the following holds: 
    \begin{itemize}
        \item The vertex $x$ is near-simplicial in $G$ (contrary to our assumption that $x$ is not near-simplicial in $\beta^*$); or
        \item $N_G(y) \setminus \{a\}$ is a clique anticomplete to $\{a\}$ (contrary to the fact that $y'$ is in $N_G(y) \setminus \{a\}$ and $y'a \in E(G)$). 
    \end{itemize}
    This is a contradiction, and proves \eqref{st:twonbrsdeg}. 

    \medskip 

    By \eqref{st:twonbrsdeg}, there is at most one vertex in $N_{\beta^*}(x) \setminus K$ of degree more than two in $\beta^*$; let us choose $x'$ to be this vertex if it exists (letting $x'$ be an arbitrary vertex in $N_{\beta^*}(x)$ otherwise). Then $Z_1 = K \cup \{x'\}$ satisfies (i). 

    It remains to prove (ii). Let $K, Z_1, x'$ be as above. Let $a, u \in N_{\beta^*}(x) \setminus Z_1$ with $a, u \in \beta$. Then, by Lemma \ref{lem:centralseagull}, at least one of $a, u$ is near-simplicial in $G$ and simplicial in $G \setminus \{x\}$. We define a vertex $x''$ as follows: If every vertex in $N_{\beta}(x) \setminus Z_1$ is simplicial in $G \setminus \{x\}$, then $x'' = x'$. Otherwise, $x''$ is the unique vertex in $N_{\beta}(x)$ which is not simplicial in $G \setminus \{x\}$. Now let $Z_2 = Z_1 \cup \{x''\} = K \cup \{x', x''\}$.

    We claim that $Z_2$ satisfies (ii). Suppose for a contradiction that (ii) does not hold, that is, there is a vertex $y \in N_{\beta^*}(x) \setminus Z_2$ and a set $X \in \core_{w, G}(\mathcal{X}(G))$ with $X = K_1 \cup K_2 = \closure(K_1, K_2)$ such that $y$ has a neighbour in a component of $D(X)$ and $x, y \in K_1$. 

    Suppose first that $y \in \beta$. From the choice of $Z_2$, it follows that $N_G(y) = \{x\} \cup T$, where $T$ is a clique anticomplete to $\{x\}$. It follows that $K_1 = \{x, y\}$. We note that $|T \cap K_2| \leq 1$, as otherwise $y$ is adjacent to all of $K_2$ (since $G$ is diamond-free), but then $X \subseteq N[y]$ is a star cutset, contradicting Assumption \ref{as:centralbag}. 
    
    Letting $(A, C, B) = S_{w, G}(X)$, and using that $T$ is a clique, we conclude that one of the following holds:
    \begin{itemize}
        \item $T \subseteq C \cup A$. Then, $y$ is not in $N(B)$, and $y$ does not have two neighbours in either $K_1$ or $K_2$, contrary to the fact that $X = K_1 \cup K_2 = \closure(K_1, K_2)$, a contradiction. 
        \item $T \subseteq C \cup B$ and $T \cap B \neq \emptyset$. Then, $X \setminus \{y\}$ is a cutset separating the connected component $B \cup \{y\}$ from $A$. Moreover, $X \setminus \{y\}$ is the union of the two cliques $K_2$ and  $K_1 \setminus \{y\} = \{x\}$, and $X' = \closure(\{x\}, K_2) \subseteq X \setminus \{y\}$; therefore, $X' \in \mathcal{X}(G)$ is a $(w, G)$-shield for $X$, a contradiction. 
    \end{itemize}

    It follows that $y \in \beta^* \setminus \beta$. Then, $y \in P(X')$ for some $X' \in \core_{w, G}(\mathcal{X}(G))$; and in particular, there is a component $D'$ in $D(X')$ such that $y \in D'$ and $D'$ is a component of $G \setminus \beta$. Since $y$ has a neighbour in a component $D$ of $D(X)$, and $D$ is a component of $G \setminus \beta$, it follows that $D = D'$. But then, once again letting $(A, C, B) = S_{w, G}(X)$, we conclude that $y \in A$, a contradiction as we had assumed that $y \in K_1 \subseteq C$. This proves (ii). 
\end{proof}

\section{Putting everything together}

We require the following result and definition of \cite{abrishami2022induced}. For a graph $G$ and positive integer $d$, we denote by $\gamma_d(G)$ the maximum
degree of the subgraph of $G$ induced by the set of vertices with degree at least $d$ in $G$.

\begin{theorem}[\previousauthors]
  \label{smalldensecomps}
  For all $t, \gamma > 0$, there exists $q = q(t, \gamma)$ such that every
graph $G$ with $\gamma_3(G) \leq \gamma$
and treewidth more than $q$ contains a subdivision of $W_{t\times t}$
or the line graph of a subdivision of $W_{t\times t}$ as an induced subgraph.
\end{theorem}

Now we can prove: 
\begin{lemma} \label{lem:sepbeta}
    For every $t \in \mathbb{N}$, there exists a constant $n = n(t)$ such that the following holds. 
    
    Let $G$ be a $t$-clean graph and assume that Assumption \ref{as:centralbag} holds. With the definition of $\beta^*$ and $w^*$ as in the previous section, we have that $\beta^*$ has a $\left(w^*, \frac{1}{2}\right)$-balanced separator of size at most $n$. 
\end{lemma}
\begin{proof}
   Let $q$ be as in Theorem \ref{smalldensecomps} applied with $t$ and with $\gamma = t$; let $n = q+1$. Then, by Lemma \ref{lem:centraldegree}(i), every vertex in $\beta^*$ has at most $\omega(G) \leq t$ (as $G$ is $t$-clean) neighbours of degree more than two, it follows that $\gamma_3(\beta^*) \leq t$. Therefore, by Theorem \ref{smalldensecomps}, and since $G$ is $t$-clean, it follows that $\tw(\beta^*) \leq q$. Now, by Lemma \ref{lemma:tw-to-weighted-separator}, it follows that $\beta^*$ has a $\left(w^*, \frac{1}{2}\right)$-balanced separator of size at most $q+1 = n$. 
\end{proof}

Our final step is to lift the balanced separator from $\beta^*$ to $G$. Given two sets $X, Y \subseteq V(G)$, let us say that $X$ \emph{has a neighbour in $Y$} if there is an edge $xy \in E(G)$ with $x \in X$ and $y \in Y$. 
\begin{theorem} \label{thm:finale}
    Let $t \in \mathbb{N}$. There exists a $c = c(t)$ such that the following holds. Let $G$ be a $t$-clean (clock, diamond)-free graph with no star cutset. Let $w$ be a weight function on $G$. Then $G$ has a \whalf-balanced separator of size at most $c$. 
\end{theorem}

Recall that the above, together with Lemma \ref{lemma:bs-to-tw}, implies Theorem \ref{thm:mainnew}, which in turn implies Theorem~\ref{thm:main}, as discussed in Section~\ref{sec:diamond}. It remains to prove Theorem \ref{thm:finale}. 

\begin{proof}[Proof of Theorem \ref{thm:finale}]
    Let $n = n(t)$ be as in Lemma \ref{lem:sepbeta}. We define 
    $c = \max\{4t, n(2t+1)\}.$

    Suppose that $G$ contains no \whalf-balanced separator of size at most $c$. Since $G$ is $t$-clean, it follows that $\omega(G) \leq t$. Since $c \geq 4t$, it follows that $G$ contains no \whalf-balanced separator of size at most $4 \omega(G)$. Therefore, Assumption \ref{as:centralbag} holds, and in particular, we can define $\beta^*$ and $w^*$ as in Section \ref{sec:inside}. 

    By Lemma \ref{lem:sepbeta}, it follows that $\beta^*$ has a $\left(w^*, \frac{1}{2}\right)$-balanced separator $S$ of size at most $n$. 

    Let us define a set $Y$ as follows. For every $v \in S$ such that $v \in (P(X))^*$ for some $X \in \core_{w, S}(\mathcal{X}(G))$, we set $Y(v) = X$. For every $v \in S$ which is not a claw center in $\beta^*$, we let $Y(v) = N_{\beta^*}(v)$. For every $v \in S$ which is a claw center in $\beta^*$, we let $Y(v) = Z_2$ where $Z_2$ is as is Lemma \ref{lem:centraldegree}(ii). Finally, we let $Y = S \cup \bigcup_{v \in S} Y(v)$. From the definition of $Y(v)$ above, it follows that $Y(v)$ is either the union of two cliques or the union of a clique and two vertices for every $v \in S$; therefore, $|Y| \leq |S|(2t+1) \leq n(2t+1) \leq c$, as desired. 

    It remains to show that $Y$ is a \whalf-balanced separator of $G$. Let $D$ be a component of $G \setminus \beta$. Writing $X = X(D)$ and $S_{w, G}(X) = (A, C, B)$, we have that $D \subseteq A$ and so $w(D) \leq \frac{1}{2}$. Let $D'$ be a component of $D \setminus Y$. If $D'$ has no neighbour in $\beta \setminus Y$, then $D'$ is a component of $G \setminus Y$ and of weight at most $\frac{1}{2}$, as desired. 
    
    Given a component $R$ of $\beta \setminus Y$, let us define $\hat{R}$ as the component of $\beta^* \setminus S$ containing $R$. 
    
    \sta{\label{st:almostdone} Let $D$ be a component of $G \setminus \beta$, and let $D'$ be a component of $D \setminus Y$. Let $X = X(D)$. If $D'$ has a neighbour in a component $R$ of $\beta \setminus Y$, then $a_X \in \hat{R}$. In particular, $\hat{R}$ is the same for all components $R$ of $\beta \setminus Y$ in which $D'$ has a neighbour.}
    
    Suppose not, that is, $D'$ has a neighbour $y$ in a component $R$ of $\beta \setminus Y$, but $a_X \not\in \hat{R}$. If $(P(X))^* \cap S \neq \emptyset$, say $q \in (P(X))^* \cap S$, then $X = Y(q) \subseteq Y$, and so $D'$ has no neighbour in $\beta \setminus Y \subseteq \beta \setminus X$. Therefore, we may assume that $a_X \in (P(X))^*$ is in a component $R'$ of $\beta^* \setminus S$ with $R' \neq \hat{R}$. 

    Writing $X = K_1 \cup K_2 = \closure(K_1, K_2)$, and using that $y \in \beta \cap N_G(D') \subseteq \beta \cap N_G[D] = X$, we may assume that $y \in K_1$. Let $x, x'$ be the ends of $P(X)$. Since $(P(X))^*$ is disjoint from $S$ and contains $a_X \in R'$, it follows that $(P(X))^* \subseteq R'$. But one of $x, x'$, say $x$, is in $K_1$ and therefore adjacent to $y \in \hat{R}$ (which is anticomplete to $R'$); so we conclude that $x \in S$ (as $x$ has a neighbour in both $\hat{R}$ and $R'$, two different components of $\beta^* \setminus S$); see Figure \ref{fig:lifting}. But then, as $y \in N_{\beta^*}(x) \setminus Y(x)$, it follows that $x$ is a claw center in $\beta^*$, and by Lemma \ref{lem:centraldegree} and the choice of $Y(x)$, there is no $X \in \core_{w, G}(\mathcal{X}(G))$ such $X = K_1 \cup K_2 = \closure(K_1, K_2)$ such that $x, y \in K_1$ and $y$ has a neighbour in $D' \subseteq D(X)$. This is a contradiction (because $X, x, y, K_1, K_2$ have precisely those properties), and proves \eqref{st:almostdone}. 

    \medskip

    \begin{figure}[t]
    \begin{center}
        \includegraphics[width=0.5\textwidth]{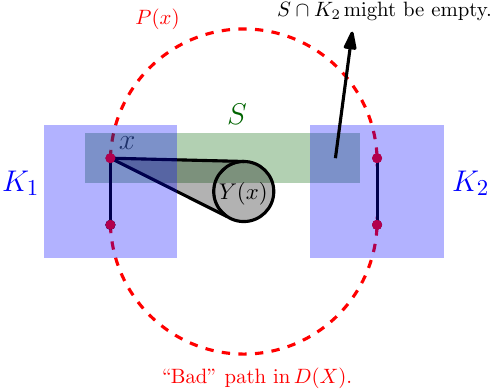}
        \caption{Proof of Theorem \ref{thm:finale}. \pictured}
        \label{fig:lifting}
        \end{center}
    \end{figure}

    From \eqref{st:almostdone}, it follows that for every component $M$ of $G \setminus Y$, there is a component $\hat{R}$ of $\beta^* \setminus S$ such that: 
    \begin{itemize}
        \item $M \cap \beta \subseteq \hat{R}$; and
        \item For every $x \in G \setminus \beta$ such that $x$ is in a component $D$ of $G \setminus \beta$ and $x \in M$, we have that $a_{X(D)} \in \hat{R}$. 
    \end{itemize}
    Now it follows that $w(M) \leq w^*(\hat{R}) \leq \frac{1}{2}$; but then $Y$ is a \whalf-balanced separator, as desired. 
\end{proof}
 
\bibliographystyle{plain}
\bibliography{bib}

\end{document}